\documentclass[12pt]{article}
\usepackage[T1,T2A]{fontenc}
\usepackage[english]{babel}
\usepackage{amsfonts,amsmath,amssymb,amsthm,graphicx,afterpage,url}
\input{xy}
\xyoption{all}
\graphicspath{{invasurf_figures/}}

\newcommand{\jonly}[1]{}
\newcommand{\arxiv}[1]{#1}

\usepackage{hyperref}
\hypersetup{
   colorlinks   = true, 
   urlcolor     = blue, 
   linkcolor    = blue, 
   citecolor   = red 
}

\def\diag{\mathop{\fam0 diag}}

\def\Hom{\mathop{\fam0 Hom}}
\def\Ext{\mathop{\fam0 Ext}}

\def\rk{\mathop{\fam0 rk}}

\def\Int{\mathop{\fam0 Int}}

\def\t{\widetilde}

\def\R{{\mathbb R}} \def\Z{{\mathbb Z}} \def\Hc{{\mathbb H}}  \def\C{\Bbb C} \def\Q{\Bbb Q}

\let\Bbb=\mathbb

\theoremstyle{theorem}
\newtheorem{theorem}{Theorem}[section]
    \newtheorem{lemma}[theorem]{Lemma}
    \newtheorem{corollary}[theorem]{Corollary}
    \newtheorem{proposition}[theorem]{Proposition}
    \newtheorem{addendum}[theorem]{Addendum}

\theoremstyle{definition}

\newtheorem{remark}[theorem]{Remark}
\newtheorem{example}[theorem]{Example}

\newtheorem{problem}[theorem]{Problem}

\begin{document}

\title{Embeddings of $k$-complexes in $2k$-manifolds, and minimum rank of partial symmetric matrices\footnote{I am grateful for useful discussions to A. Bikeev, S. Dzhenzher, R. Fulek, T. Garaev, R. Karasev, A. Kliaczko, E. Kogan, M. Tancer, I. Zhiltsov, and anonymous referees. 
I am grateful to anonymous referees and to AI for finding flaws in the previous versions of this paper. 
Lemma \ref{l:com_div} is written by the author using a discussion with AI. 
Supported by Russian Science Foundation, Grant N 25-21-00685.}}

\author{A. Skopenkov}

\date{}

\maketitle

\begin{abstract}
Let $K$ be a $k$-dimensional simplicial complex having $n$ faces of dimension $k$,
and $M$ a closed $(k-1)$-connected PL $2k$-dimensional manifold.
We prove that {\it for $k\ge3$ odd, $K$ embeds into $M$ if and only if there are

$\bullet$ a skew-symmetric $n\times n$ matrix $A$  with integer entries, whose rank over $\Q$ does not exceed
$\rk H_k(M;\Z)$,  

$\bullet$ a general position PL map $f:K\to\R^{2k}$, and

$\bullet$ orientations on $k$-faces of $K$

such that for any disjoint $k$-dimensional faces $\sigma,\tau$ of $K$ the entry $A_{\sigma,\tau}$ equals the algebraic intersection of $f\sigma$ and $f\tau$.} 

We prove some analogues of this result (for any parity of $k$), including those for $\Z_2$- and $\Z$-embeddability.
Our results generalize the Bikeev-Fulek-Kyn\v cl
criteria for the $\Z_2$- and  $\Z$-embeddability of graphs to surfaces, and are related to the Harris-Krushkal-Johnson-Pat\'ak-Tancer criteria for the embeddability of $k$-complexes into $2k$-manifolds.
The main novelty of this paper is passing from the cohomology condition of Pat\'ak-Tancer to the simpler 
\emph{extendability of some intersection function to a low-rank matrix} (defined in the paper using the idea of Fulek-Kyn\v cl).
\end{abstract}

\jonly{\noindent

}

\tableofcontents

\section{Introduction and main results}\label{ss:introd}

{\bf Informal description of main results.}

The study of graph drawings on 2-dimensional surfaces is an active area of mathematical research.
Higher-dimensional generalization is classical, and has attracted some attention recently, see Remark \ref{r:mot}.bcf.
In this paper a {\bf $k$-complex} is a collection of some closed faces of some simplex, closed under taking subfaces, and whose maximal dimension is $k$. 
(a longer term is $k$-dimensional finite simplicial complex).
We identify this set with the union of those faces (i.e. with the \emph{body} of the $k$-complex).
Just like for $k=1$, it is known and simple that any $k$-complex embeds into {\it some} $2k$-dimensional manifold.
A major problem asks if there exists an algorithm for recognizing the embeddability of $k$-complexes into a {\it given} $2k$-dimensional manifold.
A quick algorithm based on a beautiful mathematical result is (as always) preferable.
There are classical algorithms for $k=1$ and an arbitrary 2-manifold, or for $k\ge3$ and the manifold being a sphere (or a ball or $\R^{2k}$), see e.g. survey \cite[\S1]{MTW}.
For $k=2$, or for $k\ge3$ and a closed manifold different from $S^{2k}$, no algorithm is known.

Our main results are criteria for embeddability (and for the $\Z_2$- and $\Z$-embeddability defined below)
of $k$-complexes to $2k$-dimensional manifolds (Corollaries \ref{c:oddk} and \ref{c:su-mohi}, Theorem \ref{t:matrix}).
We reduce embeddability to finding minimal rank of a matrix whose entries can be changed subject to some linear conditions.
This is a version of `low rank matrix completion problem' related to the Netflix problem from machine learning, and extensively studied in computer science, see e.g. \cite{Ko21}, surveys \cite{MC, NKS}, and the introduction \cite{DGN+} accessible to students.
 
These criteria allow to prove an additivity result (Corollary \ref{c:rkcom}.b) and \cite[Theorem 2.4.2]{KS21e} due to E. Kogan \cite{Ko}.
We also present Corollary \ref{c:rkcom}.ac proved using an essentially known criterion.


\bigskip
{\bf Notation and conventions.}

In this paper, unless otherwise indicated, we consider only piecewise linear (PL) manifolds and maps (thus we mostly omit `PL').
The analogues of
Corollaries \ref{c:oddk} and \ref{c:su-mohi}.ab are correct for topological embeddings (by the PL approximation theorem \cite[Theorem 1]{Br72}\arxiv{, cf. \cite[Remark 1.4.b]{DS22}}).
We shorten `compact manifold, possibly with boundary' to `manifold', and `$k$-dimensional (face, manifold, etc)' to `$k$- (face, manifold, etc)'.
Further, $k$ is any positive integer, $K$ is any $k$-complex, $n=n_K>0$ is the number of $k$-faces (for $k=1$ edges) of $K$, and faces are closed.
Fix any orientations on $k$-faces of $K$. 
(The existential matrix criteria below are independent of this choice.) 

A map $f\colon K\to M$ to a $2k$-manifold is said to be a \textbf{general position map} if

$\bullet$ any two vertex-disjoint faces the sum of whose dimensions is less than $2k$ have disjoint images,

$\bullet$ the restriction of $f$ to any $k$-face has a finite set of self-intersection points,

$\bullet$ for any vertex-disjoint $k$-faces $\sigma,\tau$

\quad --- the set $f\sigma \cap f\tau$ is finite and is disjoint with self-intersections of $f|_\sigma$ and $f|_\tau$,

\quad --- for any point $y\in f\sigma \cap f\tau$, and some small $(2k-1)$-sphere $S$ centered at $y$, 
the intersections $S\cap f\sigma$ and $S\cap f\tau$ are $(k-1)$-spheres for which the absolute value of linking number in $S$ is $1$.

\begin{figure}[ht]
\centerline{\includegraphics[width=4.5cm]
{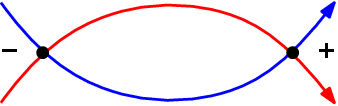} \qquad
\includegraphics{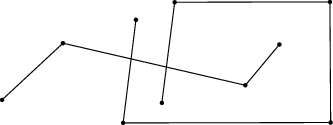}}
\caption{Two curves intersecting at an even number of points the sum of whose intersection signs is zero (left) or non-zero (right).}
\label{f:curves}
\end{figure}

If 
$M$ is oriented, then the $\pm1$ linking number is defined; it is called \textit{the intersection sign} of $y$ 
(Figure \ref{f:curves}); the linking number does not depend on the above sphere $S$.

Denote by $\Delta_N$ the $N$-simplex, and by $\Delta_N^j$ the union of the $j$-faces of $\Delta_N$.

Let $h:\Delta_k\sqcup \Delta_k'\to M$ be a general position map to an oriented $2k$-manifold, where $\Delta_k'$ is a copy of $\Delta_k$.
The \textbf{algebraic intersection number} $h\Delta_k\cdot h\Delta_k'\in\Z$ is
the sum of the intersection signs of all intersection points from $h\Delta_k\cap h\Delta_k'$.
Observe that $h\Delta_k\cdot h\Delta_k' = (-1)^kh\Delta_k'\cdot h\Delta_k$.

Denote by $\t K$ the set of all ordered pairs of non-adjacent $k$-faces of $K$.
Functions $\t K\to\Z$ can be regarded as `partial matrices', i.e. as arrangements of integers in those cells of an $n\times n$ matrix that correspond to the pairs from $\t K$.
A function $\varphi:\t K\to\Z$ is called an \textbf{intersection function for $K$} if there is a general position map $f:K\to\R^{2k}$ such that 
$$\varphi(\sigma,\tau)=f\sigma\cdot f\tau\quad\text{for every}\quad(\sigma,\tau)\in\t K$$ 
(for the fixed orientations on $\sigma,\tau$). 
Observe that any intersection function for $K$
is \emph{$(-1)^k$-symmetric}, i.e. $\varphi(\sigma,\tau) = (-1)^k\varphi(\tau,\sigma)$ for every $(\sigma,\tau)\in\t K$.

\begin{proposition}\label{p:intfun}
There are functions $\nu_0,\delta_1,\ldots,\delta_s:\t K\to\Z$ effectively calculated from $K$, and such that
$\varphi:\t K\to\Z$ is an intersection function for $K$ if and only if there are integers $n_1,\ldots,n_s$ for which $\varphi=\nu_0+n_1\delta_1+\ldots+n_s\delta_s$.
\end{proposition}

This holds by the case $M=\R^{2k}$ of Lemma \ref{l:zstar}.
Propositions \ref{p:intfun} and \ref{p:addi} are simple, so they should be considered known.
By Proposition \ref{p:intfun} it is algorithmically decidable if given function $\varphi:\t K\to\Z$ is an intersection function for given $K$.

A complex $X$ is said to be \emph{$s$-connected} if for every $j=0,1,\ldots,s$ any continuous map
$S^j\to X$ of the $j$-sphere extends to a continuous map $D^{j+1}\to X$ of the $(j+1)$-disk.

Let $G$ be either $\Z_2$ or $\Z$, and $M$ a $2k$-manifold.
For simple (accessible to non-specialists in topology) definitions of
the \textit{homology group} $H_k(M;G)$ and the \textit{intersection form}
$\cap_{M;G} : H_k(M;G)\times H_k(M;G) \to G$ see \cite[\S1.1, \S5.3]{Pr07}, \cite[\S10.6, \S10.7]{Sk20}, \cite{HG}, \cite[\S2]{IF}. 
We omit $\Z_2$-coefficients from the notation of homology groups and intersection forms.
By the rank
of an integer matrix, or of a bilinear form on $\Z^a$, we mean its rank over $\Q$.
We denote the ranks of the intersection forms by $\rk M$  and $\rk_\Z M$.\footnote{\label{f:rank} E.g. if $M$ is the connected sum of $s$ copies of $S^k\times S^k$, then $\rk M=\rk_\Z M=2s$; also \linebreak $\rk \C P^2 = \rk_\Z\C P^2 = \rk\Hc P^2 = \rk_\Z\Hc P^2 = 1$.}


\bigskip
\textbf{Main results for embeddings and $\Z$-embeddings.}


\begin{corollary}\label{c:oddk}
Let $k\ge3$ be an odd integer, and $M$ a $(k-1)$-connected $2k$-manifold.

(a) Assume that $M$ is closed. 
The complex $K$ embeds into $M$ if and only if some intersection function for $K$ extends to an $n\times n$ skew-symmetric integer matrix
whose rank does not exceed $\rk_\Z M$.

(b) Assume that $M$ has boundary, and take the normal form coefficients $m_1,\ldots,m_{\rk_\Z M/2}$ of $\cap_{M;\Z}$ (they are defined in Lemma \ref{l:in_sk_no_fo}; they are well-defined for the form not just for its matrix). 
The complex $K$ embeds into $M$ if and only if some intersection function for $K$ extends to an $n\times n$ skew-symmetric integer matrix 
$A$ such that $\rk A\le \rk_\Z M$ and $m_i\mid a_i$ for every $1\le i\le\rk A/2$ and the normal form coefficients $a_1,\ldots,a_{\rk A/2}$ of $A$. 
\end{corollary}

Part (a) follows by Theorems \ref{t:matrixz} and \ref{t:vkswu}.a, skew-symmetry of $\cap_{M;\Z}$ for odd $k$, the fact that \emph{any unimodular skew-symmetric bilinear form over $\Z$ is isomorphic to the symplectic form} \cite{MH73} (see also \cite[Proposition 6.2]{IF} and Lemma \ref{l:com_div} for $b_1=\ldots=b_s=1$), and Lemma \ref{l:Z_Hg}.

Part (b) follows analogously: instead of the fact and Lemma \ref{l:Z_Hg} use Lemmas \ref{l:in_sk_no_fo} and \ref{l:com_div}. 

The analogue of Corollary \ref{c:oddk}.a holds for $k=1$, without the assumption that $M$ is closed, but with the orientability assumption, and with embeddability replaced by $\Z$-embeddability (the proof is the same; use classification of 2-manifolds; do not use Theorem \ref{t:vkswu}.a; cf. Remark \ref{r:mot}.a). 

For an orientable $2k$-manifold $M$, a general position map $h:K\to M$ is called a {\bf $\Z$-embedding} if
$h\sigma\cdot h\tau=0$ for any pair $\sigma,\tau$ of non-adjacent $k$-faces.
(The sign of $h\sigma\cdot h\tau$ depends on an arbitrary choice of orientations of $M,\sigma,\tau$,
and on the order of $\sigma,\tau$, but the condition $h\sigma\cdot h\tau=0$ does not.)
For motivation and discussion of $\Z$-embeddings (and of $\Z_2$-embeddings defined below) see \S\ref{ss:alem}.
  
\begin{corollary}\label{c:su-mohi} (a) An 8-complex $K$ embeds into the octonionic projective space $\mathbb O P^2$ if and only if the following holds.

\qquad($EK$) Some intersection function for $K$ extends to the matrix $xx^T$ for some vector $x\in\Z^n$ (standardly considered as an $n\times1$ matrix).


(b) A 4-complex $K$ embeds into the quaternionic projective space $\Hc P^2$ if and only if ($EK$) holds.


(c) A 2-complex $K$ has a $\Z$-embedding 
to the complex projective space $\C P^2$ if and only if ($EK$) holds.
\end{corollary}

Corollary \ref{c:su-mohi} follows by Theorem \ref{t:matrixz} 
(for (a,b) we also need Theorem \ref{t:vkswu}.a).

It is not clear whether the condition of Corollaries \ref{c:oddk} and \ref{c:su-mohi} are algorithmically decidable (cf.  Proposition \ref{p:intfun}).
Still, the corollaries could hopefully be useful for algorithmic questions \cite[Problem 26]{PT19}. 

\smallskip
{\bf Comment.} (written by the author using a discussion with AI) For $k$ even, from a $k$-complex \(K\) one can effectively compute homogeneous quadratic polynomials $F_1,\ldots,F_r\in\Z[x_1,\ldots,x_n]$, an integer $e\ge0$, and a finite set $C\subset\Z_{2^e}^n$ such that (EK) is equivalent to
 
(*) there is $x\in\Z^n$ such that $F_1(x)=\cdots=F_r(x)=0$ and $x\bmod 2^e\in C$. 

\emph{Proof.} For $x\in\Z^n$ define a function $q(x):\t K/\Z_2\to\Z$ by $q(x)\{\sigma,\tau\} := x_\sigma x_\tau$. 
Proposition \ref{p:intfun} effectively gives functions $\nu,\delta_1,\ldots,\delta_s:\t K/\Z_2\to\Z$ such that for $m=|\t K|/2$ and
the lattice \(\Delta\subset\Z^m\) generated by $\delta_1,\ldots,\delta_s$ the property (EK) is equivalent to 


(**) `$q(x)-\nu\in\Delta$ for some $x\in\Z^n$'. 
 
We have $\Z^m/\Delta \cong \Z^r\oplus T_2\oplus T_{\mathrm{odd}}$, where \(T_2\) is a finite \(2\)-group and \(T_{\mathrm{odd}}\) has odd order. 
(Observe that $\Z^m/\Delta\cong H^{2k}(\t K/\Z_2;\Z)$, the isomorphism carrying $[\nu]$ to \emph{the van Kampen obstruction} $v(K)$.)
Reflection of \(\R^{2k}\) negates an intersection function. 
Consequently, $2\nu\in\Delta$. 
Hence \([\nu]\in\Z^m/\Delta\) has zero free and odd-primary components.
The free coordinates of \([q(x)]\in\Z^m/\Delta\) are effectively computable homogeneous quadratic polynomials \(F_1(x),\ldots,F_r(x)\). 
Choose an integer $e\ge0$ such that \(2^eT_2=0\). 
Take any projection $p_2:\Z^m/\Delta\to T_2$. 
The condition `$p_2[q(x)]=p_2[\nu]$' depends only on $x\bmod 2^e$; enumerate these residues to obtain 
$C := \{x\bmod 2^e\in\Z_{2^e}^n\ :\ p_2([q(x)])=p_2([\nu])\}$.
 
Clearly, (**) $\Rightarrow$ (*). 
Let us prove (*) $\Rightarrow$ (**). 
Suppose $x\in\Z^n$, $F_1(x)=\cdots=F_r(x)=0$ and $x\bmod 2^e\in C$. 
Choose, by the Chinese remainder theorem, a positive integer \(\lambda\) such that $\lambda\equiv1\pmod{2^e}$ and $\lambda\equiv0\pmod{|T_{\mathrm{odd}}|}$. 
Since \(q(\lambda x)=\lambda^2q(x)\), scaling by \(\lambda\) preserves the required \(2\)-primary component, annihilates the odd-primary component, and preserves the zero free component. 
Hence $q(\lambda x)-\nu\in\Delta$. \(\square\)


\bigskip
\textbf{Main result for $\Z_2$-embeddings.}

For a $2k$-manifold $M$, a general position map $h:K\to M$ is called a {\bf $\Z_2$-embedding} if
$|h\sigma\cap h\tau|$ is even for any pair  $\sigma,\tau$ of non-adjacent $k$-faces.
Clearly, any $\Z$-embedding is a $\Z_2$-embedding.
A $\Z_2$-embedding which is not a $\Z$-embedding is exhibited in Fig. \ref{f:curves}, right.

Denote by $|S|_2\in\Z_2$ the number of elements modulo 2 in a finite set $S$.


Denote by $\t K/\Z_2$ the set of all non-ordered pairs of non-adjacent $k$-faces of $K$.
Functions $\t K/\Z_2\to\Z_2$ can be regarded as `symmetric partial matrices'.
A function $\varphi:\t K/\Z_2\to\Z_2$ is called a \textbf{modulo 2 intersection function for $K$} if there is a general position map $f:K \to\R^{2k}$ such that $$\varphi(\sigma,\tau)=|f\sigma\cap f\tau|_2\quad\text{for every}\quad\{\sigma,\tau\}\in\t K/\Z_2.$$
There are functions $\nu_0,\delta_1,\ldots,\delta_s:\t K/\Z_2\to\Z_2$ effectively calculated from $K$ and
such that $\varphi:\t K/\Z_2\to\Z_2$ is an intersection function for $K$ if and only if there are $n_1,\ldots,n_s\in\Z_2$ such that $\varphi=\nu_0+n_1\delta_1+\ldots+n_s\delta_s$.
(Analogously to Proposition \ref{p:intfun}, the equivalence follows from the case $M=\R^{2k}$ of Lemma \ref{star}.)
Thus it is algorithmically decidable if given function $\varphi:\t K/\Z_2\to\Z_2$ is a modulo 2 intersection function for given $K$.

For $G=\Z$ or $\Z_2$ a bilinear form $q:V\times V\to G$ on a $\Z_2$-vector space or on a $\Z$-module $V$  is called \textbf{even} if $q(v,v)$ is even for every $v\in V$, and is called \textbf{odd} otherwise.
A symmetric matrix with $\Z_2$- or $\Z$-entries is called \textbf{even} (for the case of $\Z_2$ a.k.a. \emph{alternate}) if its diagonal contains only even entries, and is called \textbf{odd} otherwise.
The \textbf{type} of a  bilinear form or of a symmetric matrix is its being even or odd.

We shorten `matrix with $\Z_2$-entries' to just `matrix'.

\begin{theorem}[proved in \S\ref{ss:proof}, \S\ref{ss:real}]\label{t:matrix}
Let $M$ be a $(k-1)$-connected $2k$-manifold.
There is a $\Z_2$-embedding $K\to M$ if and only if
some modulo 2 intersection function for $K$ extends to a symmetric $n\times n$ matrix
of the same type as $\cap_M$, and whose rank does not exceed $\rk M$.
\end{theorem}

See more equivalent conditions in
\cite[Proposition 2.5.1]{KS21e}, \cite{Ko}.
The connectedness assumption in Theorem \ref{t:matrix} is essential for $k=1$\jonly{ \cite[footnote 2]{Sk24}}.\arxiv{\footnote{Let us prove that. Since
$K_{10}$ contains $K_5\sqcup K_5$, by Proposition \ref{p:addi}.b we have $r_{\Z_2}K_{10}\ge2$.
Then $K_{10}$ has no $\Z_2$-embedding to the Moebius band.
So $K_{10}$ has no $\Z_2$-embedding to the disjoint union $M$ of 1000 Moebius bands.
However, $K_{10}$ has $10\cdot9/2=45$ edges, so it embeds into the disk with $45\cdot 44/2<1000$ Moebius bands.
Since $K_{10}$ has 45 edges, any its intersection function extends to
some odd matrix of rank at most $45<1000=\rk M$.}}
It would be interesting to know if the connectedness assumption in Theorems \ref{t:matrix} and \ref{t:matrixz} is essential for $k>1$.

\bigskip
\textbf{Criteria in terms of low rank Gramian matrices.}

For even $k$ and general $M$ an integer analogue of Theorem \ref{t:matrix} is more complicated than Corollaries \ref{c:oddk} and \ref{c:su-mohi}.
 
The Gramian matrices of elements from $H_k(M;\Z)$ or from $H_k(M)$ are considered with respect to $\cap_{M;\Z}$ or  to $\cap_M$, respectively.

\begin{theorem}[proved in \S\ref{ss:proof}, \S\ref{ss:real}; see Remark \ref{r:mot}.c]\label{t:matrixz}
Let $M$ be a $(k-1)$-connected orientable $2k$-manifold.
There is a $\Z$-embedding $K\to M$ if and only if some intersection function for $K$ extends to a Gramian matrix of some $n$ homology classes in $H_k(M;\Z)$ (indexed by $k$-faces of $K$).
\end{theorem}

The condition of Theorem \ref{t:matrixz} is equivalent to solvability of a quadratic system of Diophantine equations.
(Indeed, the variables are numbers $n_1,\ldots,n_s$ from Proposition \ref{p:intfun}, and integer coefficients in expressions of the homology classes through a base of $H_k(M;\Z)$, cf. a discussion in \cite[\S1.3, the 2nd paragraph after Question 9]{PT19}.)
It is unclear if this solvability is algorithmically decidable.

\begin{corollary}\label{c:oddind} Assume that $k$ is even, $M$ is a closed orientable $(k-1)$-connected $2k$-manifold, and $\cap_{M;\Z}$ is odd indefinite, with positive and negative ranks $r_+$ and $r_-$.
There is a $\Z$-embedding $K\to M$ if and only if some intersection function for $K$ extends to a matrix $A_{\sigma,\tau}:=x_\sigma^Tx_\tau-y_\sigma^Ty_\tau$ for some $n$ vectors $x_\sigma\in\Z^{r_+}$ and $n$ vectors $y_\sigma\in\Z^{r_-}$.
\end{corollary}

This follows by Theorem \ref{t:matrixz} and classification of symmetric unimodular odd indefinite bilinear forms  over $\Z$ \cite{MH73} (see also \cite[Theorem 6.3.a]{IF}).

It would be interesting to obtain a version of Corollary \ref{c:oddind} analogous to Corollaries \ref{c:oddk} and \ref{c:su-mohi}, i.e. in terms of the intersection form of $M$, without mentioning Gram matrix of certain homology classes. 

\bigskip
\textbf{Additivity.}

The \emph{rank} $P(K)$ is the minimal number $\rk M$, where $M$ is a $2k$-manifold admitting an embedding $K\to M$.
The \emph{$\Z_2$-rank} $R_{\Z_2}(K)$ is the minimal number $\rk M$, where $M$ is a $2k$-manifold admitting a $\Z_2$-embedding $K\to M$.
The \emph{$\Z$-rank} $R_{\Z}(K)$ is the minimal number $\rk_\Z M$, where $M$ is an orientable $2k$-manifold admitting a $\Z$-embedding $K\to M$.

\begin{proposition}[additivity]\label{p:addi} For any $k$-complexes $X,Y$ we have

(a) $P(X\sqcup Y)=P(X)+P(Y)$;

(b) $R_{\Z_2}(X\sqcup Y)=R_{\Z_2}(X)+R_{\Z_2}(Y)$;

(c) $R_\Z(X\sqcup Y)=R_\Z(X)+R_\Z(Y)$.
\end{proposition}

Parts (b,c) for $k\ge2$ and part (a) are proved in \S\ref{ss:coro} independently of any embeddability criteria; the proof appeared in a discussion with T. Garaev.
Parts (b,c) for $k=1$ are essentially known (see Remark \ref{r:mot}); they follow from Corollary \ref{c:rkcom}.bc (additivity) because if $K$ embeds into $M\sqcup N$, then $K$ embeds into $M\#N$, so $P(K)=\rho(K)$, and analogously $R_G(K)=r_G(K)$ for $G\in\{\Z,\Z_2\}$ (see just below).


Define $\rho(K),r_{\Z_2}(K),r_\Z(K)$ analogously to $P(K),R_{\Z_2}(K),R_\Z(K)$ but adding `$(k-1)$-connected' before `$2k$-manifold'.

A \emph{maximal $k$-forest} $T\subset K$ is a subcomplex $T\subset K$ inclusion-maximal among subcomplexes containing no non-empty $k$-cycles modulo 2. 
A \emph{maximal integral $k$-forest} $T\subset K$ is a subcomplex $T\subset K$ inclusion-maximal among subcomplexes containing no non-zero integer $k$-cycles. 
The complex $K$ is said to be \emph{forestable} \emph{(integrally forestable)} if the group $H_{k-1}(T;\Z)$ is free for some maximal $k$-forest $T$ (integral maximal $k$-forest $T$, respectively).  
E.g. any graph is forestable and is integrally forestable.  
\jonly{See more examples in \cite[Comment before Corollary 1.9]{Sk24}.}

\arxiv{
\medskip
{\bf Comment.} (a) There exists a 2-complex $K$ and its maximal integral 2-forest $T$ such that $H_1(X;\Z)$ is free but $H_1(T;\Z)$ is not free.
Indeed, take 
$$K := \R P^2\cup_{\R P^1}D^2\sim \R P^2/\R P^1\cong S^2\quad\text{and}\quad T:=\R P^2$$ 
(more precisely, $T:=\R P^2\cup_{\R P^1}A\sim \R P^2$, where the annulus $A$ is perforated $D^2$).
Then $T$ is a maximal integral 2-forest for $K$, the group $H_1(K;\Z)\cong\Z$ is free but  $H_1(T;\Z)\cong\Z_2$ is not free.

(b) There exists a 2-complex $K$ and its maximal 2-forest $T$ such that $H_1(X;\Z)$ is free but  $H_1(T;\Z)$ is not free.
Indeed, for every $x\in S^1$ identify points $x,xe^{2\pi i/3},xe^{4\pi i/3}\in S^1$.
Denote by $q$ the quotient map.
Let $h:\partial D^2\to q\partial D^2$ be a homeomorphism,
$$K := qD^2\cup_h D^2 \sim qD^2/q\partial D^2 \cong S^2\quad\text{and}\quad T:=qD^2.$$
Then $T$ is a maximal 2-forest for $K$, the group $H_1(K;\Z)\cong\Z$ is free but  $H_1(T;\Z)\cong\Z_3$ is not free.
}

\begin{corollary}[additivity]\label{c:rkcom} For any $k$-complexes $X,Y$ we have

(a) $\rho(X\sqcup Y)=\rho(X)+\rho(Y)$ if $X,Y$ are integrally forestable and $k\ge3$;

(b) $r_{\Z_2}(X\sqcup Y)=r_{\Z_2}(X)+r_{\Z_2}(Y)$ if $X,Y$ are forestable; \qquad

(c)  $r_\Z(X\sqcup Y)=r_\Z(X)+r_\Z(Y)$ if $X,Y$ are integrally forestable.
\end{corollary}

By Theorem \ref{t:vkswu}.a $\rho(K)=r_\Z(K)$ for $k\ge3$.
So (a) follows by (c).
Parts (b) and (c) are proved in \S\ref{ss:coro} using Theorems \ref{t:matrix} and \ref{t:matrixz}, respectively.
It would be interesting to know if Corollary \ref{c:rkcom} holds without the forestability assumptions. 



\bigskip
\textbf{Relation to known results.}

\begin{remark}\label{r:mot}
\textbf{(a)} The cases $k=1$ of Theorems \ref{t:matrix} and \ref{t:matrixz}, of Remarks \ref{r:algalem} and \ref{r:hiconn} (and of the $\Z$-analogue of Corollary \ref{c:oddk}) are known \cite[Proposition 9, Corollary 10]{FK19}, \cite[Theorems 1.1, 1.4]{Bi21}.\arxiv{\footnote{The last phrase before \S1.1 of \cite{PT19} reads \emph{...Our algebraic description in this case provides a characterization of graphs admitting an independently even drawing into a given surface.}
This is wrong since for $k=1$ the paper \cite{PT19} presents only \emph{necessary} condition
\cite[Theorem 4]{PT19}, not a characterization (because the sufficiency results
\cite[Theorem 6 and Proposition 21]{PT19} have the assumption $k\ge3$).
It is not written in \cite{PT19} that \cite[Proposition 9 and Corollary 10]{FK19} does give such a characterization, i.e. gives the analogues of \cite[Theorems 4 and 6]{PT19} for $k=1$.
\newline
In \S1.1, in the second paragraph after Theorem 1 of \cite{PT19} one reads \emph{`Theorem 1 ... seems to say something new even for $k=1$...'}.
This is wrong because the linear estimation of \cite[Theorem 1]{PT19} is weaker than the quadratic estimation of \cite[Theorem 1]{FK19} for $m=n$ (recall that $K_{2n}\supset K_{n,n}$ and any almost embedding is a $\Z_2$-embedding).}}

The case $k=1$ of Corollary \ref{c:rkcom} (equivalent to Proposition \ref{p:addi} for $k=1$ as explained in its proof) should be considered known.
This case is classical for Corollary \ref{c:rkcom}.a, was stated as \cite[Lemmas 6, 7]{SS13} for Corollary \ref{c:rkcom}.b, and is analogous for Corollary \ref{c:rkcom}.c.
The case $k=1$ of Lemma \ref{l:realtree}.a and \ref{l:realtree}.b holds by \cite{SS13, FK19} and by \cite{Bi21}, respectively.

The paper \cite{SS13} presents heuristic ideas important enough not to challenge the priority of the authors. 
However, \arxiv{Remark \ref{r:crit}.ab}\jonly{\cite[Remark 7.1ab]{Sk24}} justifies that the proof there is not \emph{reliable} 
(in the sense of \cite{Sk21d}; see \arxiv{\S\ref{s:unre}}\jonly{\cite[\S7]{Sk24}} for a discussion).
\arxiv{The paper \cite{Bi21} is unpublished.}
The proofs in \cite{SS13, Bi21} can be easily made reliable (for \cite{Bi21} perhaps just checked to be reliable) by anyone with a moderate level of mathematical maturity, given enough time and motivation.
For mature readers who do not have enough time and motivation\arxiv{ for the job of making papers by other authors reliable}, we prove in \S\S \ref{ss:proof},\ref{ss:real},\ref{ss:coro} the cases $k=1$ of Theorems \ref{t:matrix}, \ref{t:matrixz}, and of Corollary \ref{c:rkcom}.bc (without claiming priority). 


\textbf{(b)}  Criteria for the embeddability of $k$-complexes in given $m$-manifold for $2m\ge3k+3$ are given in \cite[Theorem 1 and Corollaries 6, 7, 8]{Ha69}, in terms of isovariant maps or cohomology obstructions (in certain  configuration space).
Assume further that $m=2k\ge6$.
Such criteria in a different equivalent form involving \emph{`intersection cocycle'}
are given in \cite{Jo02}; see exposition in \S\ref{ss:real}, in
\cite[Proposition 2.5.1.EH and Remark 2.5.6.b]{KS21e};
cf. \cite{La70}.

A formula for the intersection cocycle (and for its cohomology class) is given

\quad(i) implicitly in homological terms
in \cite[Theorem 3.2]{Kr00}, see explanation in
\cite[Remark 2.5.2.b]{KS21e};

\quad(ii) for $k=1$ in terms of crossing numbers in \cite[\S3.1, equality (1)]{FK19}\jonly{, see \cite[footnote 4]{Sk24}};\arxiv{\footnote{The vectors $y_\sigma$ from \cite{FK19}
and this paper is essentially the same as the homomorphism $\psi$ from \cite{PT19}.
\newline
The information (i,ii) is missing from \cite[Remark 5.d]{PT19}; the papers \cite{Kr00, FK19} are not cited in \cite{PT19} (the paper [FK19] cited in \cite{PT19} is different from \cite{FK19} cited in this paper).
The papers \cite{Kr00, FK19} are cited with explanation of their relevance to \cite{PT19} in (arXiv version 1 of)
\cite[Remarks 1.1.1.b, 1.3.7.b and 2.5.2.b]{KS21e}.
The paper \cite{KS21e} was sent to the authors of \cite{PT19} before appearance of arXiv version 1 of \cite{KS21e}, so well before appearance of \cite{PT19}.}}

\quad(iii) in terms of cohomology class $\omega(\psi)$ of certain configuration space in \cite[Theorem 4]{PT19}, see exposition in
\cite[\S2.5, Definitions required for (RH), Lemma 2.5.5]{KS21e}.

Such a formula reduces embeddability to solvability of a system of quadratic Diophantine equations, i.e. to the existence of a Gramian matrix as in Theorem \ref{t:matrixz}.
Such an embeddability criterion is the conjunction of \cite[Theorems 4, 6 and 15]{PT19}, and is explicitly stated in Theorem \ref{t:matrixz} (see also Theorem \ref{t:vkswu}.a), without using 
(iii).\arxiv{\footnote{That criterion is used in \cite[Remark 4.3]{DS22}; Theorems 1.2-1.5 of \cite{SS23} use an elaboration \cite[Proposition 2.5.1]{KS21e} (due to E. Kogan \cite{Ko}) of that criterion.}}
\jonly{See \cite[footnote 5]{Sk24}.}
Theorem \ref{t:matrixz} (and Lemma \ref{l:realtree}) are essentially results of Pat\'ak-Tancer (with a contribution by Harris-Krushkal-Johnson), although we present an explicit statement and (in \S\ref{ss:proof}, \S\ref{ss:real}) a simpler direct proof independent of \cite{PT19}; see also (d).


\textbf{(c)} The main novelty of this paper is passing from the condition of \cite{PT19} (see (iii) of (b)) to the \emph{extendability of some intersection function to a low-rank matrix}.
The latter condition is simpler because it does not involve homology classes in the manifold, only the rank and type of the intersection form of the manifold.
(Although the definition of the intersection form requires homology, results on rank and type are stated without homology, see e.g. footnote \ref{f:rank}.)

The condition of such extendability appeared in \cite{FK19, Bi21} for graphs.
Thus our results are higher-dimensional generalizations of criteria for the $\Z_2$- and $\Z$-em\-bed\-da\-bi\-li\-ty of graphs to surfaces obtained in \jonly{the unpublished} Bikeev's paper \cite{Bi21} using ideas of Fulek-Kyn\v cl-Schaefer-Stefankovi\v c \cite{SS13, FK19}.
Just as the criteria of \cite{PT19} (in terms of cohomological obstruction), these generalizations are not very hard (modulo classical techniques of geometric topology, see details in (d,f)).
Although we provide direct proofs of these generalizations, they could be deduced from
Theorem \ref{t:matrixz}, \cite[Proposition 21 and Theorem 15]{PT19} and relations between the extendability to a matrix of given rank and type, and to a Gramian matrix (algebraic Lemmas \ref{l:main} and \ref{l:Z_Hg}).
Still, we hope these generalizations are interesting since they reveal relation to the `low rank matrix completion problem', and give interesting corollaries (see the beginning of \S\ref{ss:introd}).

\textbf{(d)} The case $k=2$ is excluded from \cite[Theorems 6, 10.i, and Proposition 21]{PT19}.
However, \cite[Theorem 10.i, and Proposition 21]{PT19}, and the analogue of \cite[Theorem 6]{PT19} obtained by replacing `embedding' by `$\Z$-embedding', do hold for $k=2$.
(By (a), these results are known for $k=1$.)
The proofs could be obtained by minor changes in \cite{PT19}; simpler proofs (of the versions in terms of low-rank matrix completion problem) are presented in this paper.
 
\textbf{(e)} For the K\"uhnel problem on embeddings of $k$-skeleta of $n$-simplices into $2k$-manifolds see recent paper \cite{DS22}
and the references therein.

\textbf{(f)} (known ideas in our proofs) Our constructions of a $\Z_2$- and $\Z$-embedding from a matrix (in Theorems \ref{t:matrix} and \ref{t:matrixz}; \S\ref{ss:proof}) use known construction of a map inducing given homomorphism in homology, cf.
\cite[Lemma 2.5.3]{KS21e}\jonly{ and \cite[footnote 6]{Sk24}.}\arxiv{.\footnote{Because of using this construction, 
also used in \cite[\S5, proof of Theorem 6, step 1]{PT19},
a part of our proof is similar to \cite{PT19}.}}
Our construction of a matrix from a $\Z_2$- and $\Z$-embedding (in Theorems \ref{t:matrix} and \ref{t:matrixz}; \S\ref{ss:real}) use
the modification of a given map as in \cite[Lemma 12]{PT19}, cf.
\cite[Lemma 2.5.5]{KS21e}.

\arxiv{Geometric proofs (closer to \cite{Bi21}) of Theorems \ref{t:matrix} and \ref{t:matrixz}
could perhaps be obtained using a handle decomposition of $M$.}
\end{remark}


 

\section{Almost, $\Z_2$- and $\Z$-embeddings: some motivation}\label{ss:alem}


A map $f:K\to Y$ of a complex $K$ to a space $Y$ is called an {\it almost embedding} if the images of non-adjacent faces are disjoint, 
i.e. if $f\sigma\cap f\tau=\emptyset$ for any non-adjacent faces  $\sigma,\tau$.
 
\begin{remark}[motivations]\label{r:alem} 
(a) The notion of an almost embedding  naturally appears in

$\bullet$ geometric topology (studies of embeddings, see (b)),

$\bullet$ combinatorial geometry (Helly-type results on convex sets, see \cite{Mat97, GPP+}), and

$\bullet$ topological combinatorics (topological Radon and Tverberg theorems, see survey \cite{Sk16}).

For recent papers on almost embeddings 
see 
\cite{ST17, PT19, KS20, Al22, AM25, AMS}.


\emph{Almost isotopy} (i.e. continuous deformation through almost embeddings) is studied under the names \emph{weak homotopy} and \emph{vertex homotopy} 
for graphs in $\R^3$ since \cite{Ta94}; for a later paper involving later references see \cite{FN09}.

The related notions of $\Z_2$- and $\Z$-embedding (defined in \S\ref{ss:introd}) also naturally appear in studies of embeddings.
The notion of a $\Z_2$-embedding (a.k.a. \emph{Hanani-Tutte} drawing) is most actively studied for graph drawings on surfaces, see survey  \cite{Sc13} and \cite{SS13, FK19, Bi21}.
A more general notion (\emph{AT graph}) has been studied  since \cite{KLN}; for a recent paper involving recent references see \cite{Ky20}.
The higher-dimensional analogue of this general notion is so natural and useful that it was used in \cite[Disjunction Theorem 3.1]{Sk02} without naming it.

 

(b) \emph{Why almost embeddings are useful for embeddings?}
Some constructions of embeddings have constructions of almost or $\Z$-embeddings as a convenient intermediate step allowing to structure the proof, 
and to describe the relation to known results and methods.
E.g. almost embeddings essentially appeared in the 1932-1967 in the proof of embeddability criteria, 
see  Theorem \ref{t:vkswu}.a and survey \cite[\S4, \S8]{Sk06}.   
The notion was used and explicitly defined in \cite[\S4]{FKT}.  

Some proofs of the non-embeddability of complexes into $\R^d$ actually show that these complexes are not almost embeddable to $\R^d$, are not $\Z_2$- or $\Z$-embeddable to $\R^d$ for $d=2k$.
This is so e.g. 

$\bullet$ for the boundary of $(d+1)$-simplex (the non-almost embeddability is the topological Radon Theorem; see surveys \cite[\S1.1, $(TR_d)$]{Sk16}, \cite[\S2.2 and Theorem 3.1.5]{Sk18}), and 
 
$\bullet$ for the $k$-complex $\Delta_{2k+2}^k$ and $d=2k$ (the non-$\Z_2$-embeddability is a version of van Kampen--Flores Theorem; see surveys \cite[\S1.1, $(VKF_{2k})$]{Sk16}, \cite[\S1.4 and Theorem 3.1.6]{Sk18}).
\end{remark}

\begin{remark}\label{r:alempr} 
(a) Clearly, the property of being an almost embedding is preserved under sufficiently small perturbation of the map 
(the same holds for the properties of being a $\Z_2$-, and a $\Z$-embedding, as opposed to the property of being an embedding).
Thus by approximation of continuous maps with PL maps we observe that for complexes in manifolds


$\bullet$ a topological embedding can be approximated by a PL almost embedding;

$\bullet$ a PL or topological almost embedding can be approximated by a general position PL almost embedding;

$\bullet$ PL almost embeddability is equivalent to topological almost embeddability.

(b) Studies of \emph{embeddings} of $k$-complexes into $2k$-manifolds for $k>1$ 
are analogous to studies of \emph{$\Z_2$-embeddings} (not embeddings) of graphs to surfaces.
Indeed, Euler inequality $V-E+F\ge\chi(M)$ is not proved, and is possibly incorrect, for a $\Z_2$-embedding of a graph to a surface $M$.
Analogously, for $k>1$ a $k$-hyperplane in $\R^{2k}$ does not split $\R^{2k}$, so a direct analogue of Euler inequality is not available for $k$-complexes in $2k$-manifolds.
For $\Z_2$-embeddings of graphs to a surface instead of Euler inequality one applies the intersection form of the surface.
Criteria for embeddability of $k$-complexes into $2k$-manifolds for $k>1$ are also obtained in terms of the intersection form, see \S\ref{ss:introd}.
For an idea in some sense replacing the Euler inequality, and implementation of this idea, see \cite{Ka91}, \cite{Ad18}, \cite[\S6]{DS22}.
\end{remark}



 
\begin{theorem}\label{t:vkswu} (a) A complex of dimension $k\ne2$ is embeddable into a simply connected $2k$-manifold $M$ if and only if the complex is $\Z$-embeddable to $M$.
(See  \cite{vK32, Sh57, Wu58} and also \cite{Ha69, Jo02}.)

(b) For every $k\ge2$ there is a $k$-complex $\Z_2$-embeddable but not $\Z$-embeddable to $\R^{2k}$ \cite[Example 3.6]{Me06}.

(c) There is a 2-complex $\Z$-embeddable but not almost embeddable to $\R^4$.
(See \cite[\S3.2, \S3.3, \S4]{FKT}, \cite[Theorem 1.6]{AMSW}, \cite{Al22}.)

(d) There is a 2-complex almost embeddable but not PL embeddable to $\R^4$ \cite[Example in p. 338]{SSS}.
\end{theorem}

\begin{proof}[Comments on the proof of (a)]
Clearly, an embedding is a $\Z$-embedding.
Let us discuss the converse implication.

For $k=1$ the converse follows by the Hanani-Tutte Theorem, see survey \cite[Theorem 1.5.3]{Sk18} (because a  simply connected 2-manifold is the 2-sphere or the 2-disk).

Assume that $k\ge3$.
For $M=\R^{2k}$ the converse is proved in \cite{vK32, Sh57, Wu58}; for a simple exposition see \cite[\S2]{FKT}, \cite[\S4]{Sk06}.
The general case is proved in the same way, just note that in
\cite[Lemma 4, 5 and application of the Whitney trick in the proof of Theorem 3]{FKT} $\R^{2k}$ could be replaced by $M$.
See details in \cite[Corollary 2 and Theorem 4]{Jo02}.
The proofs of \cite{Jo02} mentioned here and in
Remark \ref{r:hajolem} are written for `smooth' maps of complexes but work for PL maps.\jonly{ See \cite[footnote 7]{Sk24}.}\arxiv{\footnote{Part (a) is stated as \cite[Proposition 7 for $M=M'$]{PT19} (in a slightly weaker form), and
implicitly proved in (known) step 3 of \cite[\S5, proof of Theorem 6]{PT19}.
(A $\Z$-embedding in the notation of \cite{PT19} is a map $f'':K\to M$ such that $\vartheta_{f''}=0$.
So that step 3 is precisely the proof of Theorem \ref{t:vkswu}.a.)
\newline
Part (a) also follows by \cite[Corollary 4 and its proof]{Ha69}, and was essentially proved in a more general situation in \cite[\S5]{Ha69}, cf.~\cite[Corollary 6 and the third paragraph after Corollary 5]{Ha69}.}}
\end{proof}
 

For more information on almost embeddings see \cite[Remarks 1.3 and 1.4, \S11]{AMS}.

\begin{remark}\label{r:crit2} Recall that

(a) there are two isomorphism classes of non-degenerate symmetric bilinear forms over $\Z_2$ of a given rank: 
odd forms and (for an even rank) even forms \cite{MH73} (see also \cite[Theorem 6.1]{IF});

(b) any smooth manifold has a unique PL structure \cite{Mu74},
so the intersection form of a smooth $2k$-manifold can be defined as the intersection form of the corresponding PL manifold;

(c) $(k-1)$-connected smooth $2k$-manifolds with odd intersection forms exist only for $k=1,2,4,8$ 
(see a folklore proof in \cite[Remark 1.2.d]{KS21}; the PL analogue presumably holds).
\end{remark}

\begin{remark}\label{r:algalem} (a) \arxiv{Algorithms for recognizing the $\Z_2$- and $\Z$-embeddability of $k$-complexes to given $2k$-manifold are interesting.
For a related problem see \cite[Remark 2.5.7]{KS21e}.
}
\emph{If $M$ is a $(k-1)$-connected $2k$-manifold, then there is an algorithm recognizing $\Z_2$-embeddability of $k$-complexes to $M$} \cite[Theorem 10.i]{PT19}, see Remark \ref{r:mot}.e.
Simpler, this follows from Theorem \ref{t:matrix} and the case $M=\R^{2k}$ of Lemma \ref{star}.



\arxiv{Is there such a polynomial (in $n$) algorithm for $k>1$?
Cf. \cite{Ko21}.
By (b) and Remark \ref{r:crit2}.a it suffices to answer this question for $M$ being a connected sum of several copies of $S^k\times S^k$ (for $\cap_M$ even) or of $\C P^2$, of $\Hc P^2$ (for $\cap_M$ odd).
For $k=1$ see \cite[Remark 1.2.b]{Bi21}.}

(b) \emph{For a $(k-1)$-connected $2k$-manifold $M$, the $\Z_2$-embeddability of given $k$-complex to $M$ depends only on the rank and the type of $\cap_M$.}
This follows from the criterion \cite[Proposition 21]{PT19} for $\Z_2$-embeddability in terms of a cohomology obstruction, see Remark \ref{r:mot}.e.
Simpler, this follows from Theorem \ref{t:matrix} and Remark \ref{r:crit2}.a.


\emph{For a $(k-1)$-connected orientable $2k$-manifold $M$, the $\Z$-embeddability of given $k$-complex to $M$ depends only on $\cap_{M;\Z}$}.
This follows by Theorem \ref{t:matrixz}, and for $k\ge3$ is essentially the same as \cite[Proposition 7]{PT19}, cf. \arxiv{(c).}\jonly{\cite[Remark 2.4.c]{Sk24}.}

\arxiv{(c) It would be interesting to know if {\it for $k\ge3$ any two $(k-1)$-connected PL $2k$-manifolds are PL homeomorphic if they have isomorphic intersection forms and boundaries PL homeomorphic to $S^{2k-1}$}
(cf. classification of  `almost closed' manifolds \cite[p. 170]{Wa62}; pre-H-spaces and  H-spaces are defined on p. 168 and 169).
This statement implies the $\Z$ case of (b) for $k\ge3$ and $M$ closed (so it implies \cite[Proposition 7]{PT19}).}
\end{remark}

\section{Linear algebraic lemmas}\label{ss:alemmas}
 
Let $V$ be any $\Z_2$-vector space, and $I\colon V\times V\to\Z_2$ any symmetric bilinear form.
We shorten`the Gramian matriх of some vectors from $V$ with respect to $I$' to `a Gramian matrix'.

\begin{lemma}\label{l:main} 
(a) Let $A$ be a Gramian matrix.
Then there is a same-size symmetric matrix coinciding with $A$ outside the diagonal, having the same type as $I$, and whose rank does not exceed $\rk I$.

(b) Let $A$ be a symmetric matrix of the same type as $I$, and whose rank does not exceed $\rk I$.
Then $A$ is a Gramian matrix.
\end{lemma}

Lemma \ref{l:main} for $\rk A=\dim V$ implicitly appeared in \cite[\S2]{Bi21}.
Lemma \ref{l:main} is deduced below from well-known Lemmas \ref{l:gram}, \ref{l:midrk}, known Lemma \ref{l:matrix_HgI}, and simple Lemma \ref{l:alternate} (which appeared in a discussion with A. Bikeev, and presumably is known).

\begin{lemma}\label{l:gram} Let $F$ be $\Z_2$ or $\Q$.
    Let $J\colon W\times W\to F$ be a symmetric bilinear form on an $F$-vector space $W$.
    If $A$ is a Gramian matrix with respect to $J$, then $\rk A \le \rk J$.
\end{lemma}

\begin{lemma}[{\cite[Lemma 2.2]{Bi21}}]\label{l:alternate}
If $I$ is odd and $A$ is an even Gramian matrix, then 
\linebreak
$\rk A \le \rk I - 1$.
\end{lemma}

\begin{proof}[Proof Lemma \ref{l:main}.a]
By Lemma \ref{l:gram} we have $\rk A \le \rk I$.

If $I$ and $A$ have the same type, then $A$ is as required.

If $I$ and $A$ have different types, then
$I$ is odd but $A$ is even.
Hence by Lemma \ref{l:alternate} $\rk A \le \rk I-1$.
Take the matrix $A'$ obtained from $A$ by replacing the entry $A_{1,1}$ with $1$.
Then $A'$ is odd, and $\rk A' \le \rk A+1 \le \rk I$ (here the first inequality holds because changing one entry cannot change the rank more than by 1).
So $A'$ is as required.
\end{proof}

\begin{lemma}\label{l:midrk}
Let $A$ and $B$ be $n\times m$ and $m\times k$ matrices, respectively.
    Then $\rk AB \le \rk B$.
    In particular, $\rk AB \le m$.
\end{lemma}

Recall that the rank of an even symmetric matrix
is even.
Denote by $H_g$ the $2g\times2g$ matrix
formed by $g$ diagonal blocks
$\begin{pmatrix}0&1\\1&0\end{pmatrix}$ and zeros elsewhere.
For a symmetric matrix $A$ let $H_A$ be the matrix $H_{\rk A/2}$ if $A$ is even, and the $\rk A\times\rk A$ identity matrix otherwise.

\begin{lemma}[{\cite[Theorem 3]{AA38}, \cite[Theorem 1]{MW69}}]\label{l:matrix_HgI}
For any symmetric $a\times a$ matrix $A$ there is a $\rk A\times a$ matrix $Y$ such that $A = Y^TH_AY$.
\end{lemma}

\begin{proof}[Proof of Lemma \ref{l:main}.b]
(This proof appeared in a discussion with E. Kogan \cite[\S2.1]{KS21e}, \cite{Ko})
Denote by $I$ the matrix of the bilinear form $I$ in some basis of $V$.
By Lemma \ref{l:matrix_HgI} for each $C = A, I$ there is a matrix $Y_C$ such that $C = Y_C^T H_CY_C$.
By Lemma \ref{l:midrk} we have $\rk Y_I \geq \rk H_IY_I \geq \rk I$.
Hence all the $\rk I$ rows of $Y_I$ are linearly independent.
Denote $m := \dim V$.
Then there is a nondegenerate $m\times m$ matrix $Y'_I$ whose first $\rk I$ rows are the rows of $Y_I$.
Denote by $H'_I$ the $m\times m$ matrix $\begin{pmatrix}H_I&0\\0&0\end{pmatrix}$.
Then $I = (Y'_I)^T H'_I Y'_I$.
Let $a$ be the size of $A$. 
Denote by $Y$ the $m\times a$ matrix obtained from $Y_A$ by adding $m-\rk A$ zeroes below every column of $Y_A$.
Then $A = Y^TH'_IY = ((Y'_I)^{-1}Y)^T I ((Y'_I)^{-1}Y)$.
Thus $A$ is the Gramian matrix of the vectors represented in the chosen basis of $V$ by the columns of the matrix $(Y'_I)^{-1}Y$.
\end{proof}


Denote by $H$ the integer $2\times2$ matrix $\begin{pmatrix}0&1\\-1&0\end{pmatrix}$. 
For an integer $u$ denote by $uH$ the matrix obtained by multiplying every element of $H$ by $u$.  
For $a\times a$ and $b\times b$ matrices $A$ and $B$ denote by $A\oplus B$ their direct sum, 
i.e. the $(a+b)\times(a+b)$ matrix formed by diagonal blocks $A,B$, and zeros elsewhere.
For an integer $k$ denote by $\oplus_k H$ the direct sum of $k$ copies of $H$, 
i.e. the integer $2k\times2k$ matrix formed by $k$ diagonal blocks $H$ and zeros elsewhere.


\begin{lemma}\label{l:Z_Hg}
For every integer skew-symmetric $a\times a$ matrix $A$ its rank $r$ is even and there is an integer $r\times a$ matrix $Y$ such that $A = Y^T(\oplus_{r/2}H)Y$.
\end{lemma}

This follows from \cite[\S5, Theorem 1]{Bo07}, and is a particular case $d_1=\ldots=d_s=1$ of Lemma \ref{l:com_div} below.


\begin{lemma}[integral skew normal form; 
{\cite[Chapter IV, Theorem IV.1, p. 57]{Ne72}}, {\cite[exercise 4 to Chapter XIV]{La02}},  {\cite[Theorem 6.1, p. 475]{Ka49}}]\label{l:in_sk_no_fo}
For every integer skew-symmetric $a\times a$ matrix $A$ its rank $r$ is even, and there are a unimodular matrix $U$ and integers $a_1,\ldots,a_{r/2}$ 
such that 
$$0<a_1\mid a_2\mid\cdots\mid a_{r/2}\quad\text{and}\quad U^TAU = a_1H\oplus\ldots\oplus a_{r/2}H\oplus0.$$ 
Such integers (called \emph{normal form coefficients}) are uniquely determined by $A$. 
\end{lemma}
   
 
 
The following lemma and its proof were written by the author using a discussion with AI.   

\begin{lemma}\label{l:com_div} Let $b_1,\ldots,b_s$ be integers such that $0<b_1\mid b_2\mid\cdots\mid b_s$. 
Denote \(B:=b_1H\oplus\cdots\oplus b_sH\),
For an integer skew-symmetric $a\times a$ matrix $A$ there is an integer $2s\times a$ matrix $Y$ such that $A = Y^TBY$ if and only if 
$t:=\rk A/2\le s$ and $b_i\mid a_i$ for every $i=1,\ldots,t$ and the normal form coefficients $a_1,\ldots,a_t$ of $A$. 
\end{lemma}


\begin{proof}[Proof of `$\Leftarrow$'] 
Define a $2s\times a$ integer matrix $Z$ by $Z_{2i-1,\,2i-1}=1$, $Z_{2i,2i}=a_i/b_i$ for $i=1,\ldots,t$, with all other entries zero. 
Then $Z^TBZ = a_1H\oplus\cdots\oplus a_tH\oplus0 = U^TAU$ for some unimodular matrix $U$, by Lemma \ref{l:in_sk_no_fo}. 
Thus \(Y=ZU^{-1}\) is an integer matrix satisfying \(A=Y^TBY\). 
\end{proof}

For an integer matrix \(A\), \emph{the sequence $s(A)=(s_1,s_2,\ldots,s_r)$, \(r=\rk A\), of its nonzero Smith invariant factors} is defined to be the unique sequence of positive integers for which $s_1\mid s_2\mid\cdots\mid s_r$ and there exist unimodular integer matrices \(P,Q\) such that $PAQ = \diag(s_1,\ldots,s_r,0,\ldots,0)$. 
The following  lemma is simple and presumably known.

\begin{lemma}\label{l:sminv}
Let $A,B,P,Q$ be any integer matrices.  

(a) If $P$ is unimodular, then $s(P^TAP)=s(A)$.  

(b) If $A\ne0$, $B\ne0$ and $s_{\rk A}(A)\mid s_1(B)$, then $s(A\oplus B)$ is the concatenation of $s(A)$ and $s(B)$. 

(c) We have $s(uA)=|u|s(A)$ for any integer $u\ne0$. 

(d) We have $s(H)=(1,1)$. 

(e) For any $i=1,\ldots,\min\{\rk A,\rk(PAQ)\}$ we have $s_i(A)\mid s_i(PAQ)$ 
\cite[Theorem 10.5, p. 482, together with footnote 6]{Ka49}. 
\end{lemma}

\emph{The implication `$\Rightarrow$' in Lemma \ref{l:com_div}} follows by Lemma \ref{l:sminv}.e  because 
$s(A)=(a_1,a_1,\ldots,a_t,a_t)$ and 
$s(B)=(b_1,b_1,\ldots,b_s,b_s)$ by Lemma \ref{l:sminv}.abcd.
 
\arxiv{\smallskip
{\bf Comment.} There is also an equivalent characterization of matrices $A$ from Lemma \ref{l:com_div}, in terms of gcds of minors. 
Let $D_0(A):=1$ and $D_i(A):=\gcd\{|\det A_{I,J}|\ :\ |I|=|J|=i\}$. 
For \(i=1,\ldots,\rk A/2\), the Smith normal form gives $D_{2i}(A)=(a_1\cdots a_i)^2$. 
Hence the condition `$b_i\mid a_i$' can be replaced by `$b_i^2D_{2i-2}(A)\mid D_{2i}(A)$'.} 
\section{Construction of $\Z_2$- and $\Z$-embeddings}\label{ss:proof}

\begin{lemma}[known]\label{l:realcl} If $M$ is a $(k-1)$-connected manifold, then any homology class in $H_k(M)$ or in $H_k(M;\Z)$ is represented by a general position map $S^k\to M$.
\end{lemma}

\begin{proof}[A known proof]
For $k=1$ the lemma is obvious, so assume that $k\ge2$.
Then $M$ is simply-connected.
Hence $H_{k-1}(M;\Z)=0$ by Hurewicz theorem.
So the reduction modulo 2 $H_k(M;\Z)\to H_k(M)$ is epimorphic by the exactness of the coefficient exact sequence.
Hence it suffices to prove the lemma for $\Z$.

Since $k\ge2$ and $M$ is a $(k-1)$-connected, the Hurewicz map $\pi_k(M)\to H_k(M;\Z)$ is epimorphic by Hurewicz theorem.
Hence any homology class $\alpha\in H_k(M;\Z)$ is represented by a continuous map $S^k\to M$.
Then by a small shift we obtain a general position map $S^k\to M$ representing $\alpha$.
\end{proof}

\begin{proof}[Proof of the implication $(\Leftarrow)$ of Theorem \ref{t:matrix}]
By Lemma \ref{l:main}.b some modulo 2 intersection function for $K$ extends to a Gramian matrix of some $n$ homology classes $y_\sigma\in H_k(M)$ indexed by $k$-faces of $K$.

Then there is a general position map $f:K\to\R^{2k}$ such that $y_\sigma\cap_M y_\tau = |f\sigma \cap f\tau|_2$ for every non-adjacent $k$-faces $\sigma, \tau$ of $K$.
Take a $2k$-ball $B\subset\Int M$.
We may assume that $fK\subset B$.
Since $M-\Int B$  is $(k-1)$-connected and the inclusion induces an epimorphism $H_k(M-\Int B)\to H_k(M)$, by Lemma \ref{l:realcl} for any $k$-face $\sigma$ the class $y_\sigma$ is represented by a general position  map $\t\sigma:S^k\to M-\Int B$.
We may assume that $\t\sigma(S^k)\cap B=\emptyset$, and that the maps $\t\sigma$ for different $\sigma$'s are in general position to each other.
Since $M$ is connected, we can take a general position map $h:K\to M$ obtained from $f$ 
by connected summation of $f|_\sigma$ and $\t\sigma$ along certain arc $l_\sigma$, for every $k$-face $\sigma$.

(*) \emph{For $k\ge2$ by general position we may assume that
$l_\sigma\cap(f\tau\cup\widetilde\tau S^k\cup l_\tau)=\emptyset$ if $\sigma\ne\tau$.
For $k=1$ we may assume that for the tube $U_\sigma\cong[0,1]\times\{-1,1\}$ along which we make connected summation for $\sigma$,
each of the intersections of $U_\sigma$ with $f\tau$, with $\widetilde\tau S^k$, and with any other tube, consists of an even number of points.}

Now the map $h$ is a $\Z_2$-embedding because for any non-adjacent $k$-faces $\sigma, \tau$ we have
$$
|h\sigma \cap h\tau|_2 \stackrel{(1)}=
|f\sigma \cap f\tau|_2+|\t\sigma S^k\cap\t\tau S^k|_2 \stackrel{(2)}=
|f\sigma \cap f\tau|_2+y_\sigma\cap_M y_\tau \stackrel{(3)}= 0.
$$
Here

$\bullet$ equality (1) holds by (*), and because
$\widetilde\sigma S^k\cap f\tau \subset \widetilde\sigma S^k\cap B=\emptyset$ and analogously
$\widetilde\tau S^k\cap f\sigma=\emptyset$;

$\bullet$ equality (2) holds because $\widetilde\sigma S^k,\widetilde\tau S^k$ represent $y_\sigma,y_\tau$, respectively;

$\bullet$ equality (3) holds by the choice of $f$.
\end{proof}

\begin{proof}[Proof of the implication $(\Leftarrow)$ of Theorem \ref{t:matrixz}]
Denote by $y_\sigma\in H_k(M;\Z)$ any given homology classes indexed by $k$-faces of $K$. 
Then there is a general position map $f':K\to\R^{2k}$ such that $y_\sigma\cap_M y_\tau = f'\sigma\cdot f'\tau$ for every non-adjacent $k$-faces $\sigma,\tau$ of $K$. 
Let $f$ be the composition of $f'$ with the reflection of $\R^{2k}$ w.r.t. a codimension 1 hyperplane. 
Then $y_\sigma\cap_{M;\Z} y_\tau = -f\sigma\cdot f\tau$ for every non-adjacent $k$-faces $\sigma,\tau$ of $K$.  
Take a $2k$-ball $B\subset\Int M$ with the orientation inherited from $M$.
We may assume that $fK\subset B$, using an orientation-preserving homeomorphism $\R^{2k}\to B$.

The remaining proof is obtained from the remaining part of the above proof (i.e. the part after `We may assume that $fK\subset B$')  
by the following changes. 
Replace `$H_k(M-\Int B)\to H_k(M)$' by `$H_k(M-\Int B;\Z)\to H_k(M;\Z)$'.  
In (*) replace `an even number of points' by `some points the sum of whose signs is zero'.
Replace `$\Z_2$-embedding' by `$\Z$-embedding'.  
Replace the displayed formula by the following: 
$$
h\sigma\cdot h\tau\stackrel{(1)}=f\sigma\cdot f\tau+\t\sigma S^k\cdot\t\tau S^k \stackrel{(2)}=
f\sigma\cdot f\tau+y_\sigma\cap_{M;\Z} y_\tau \stackrel{(3)}= 0.\qedhere
$$
\end{proof}

\section{Construction of Gramian matrices}\label{ss:real}


{\it The implication $(\Rightarrow)$ of Theorem \ref{t:matrix}} follows by Lemma \ref{l:main}.a and Lemma \ref{l:realtree}.a.
{\it The implication $(\Rightarrow)$ of Theorem \ref{t:matrixz}} follows by Lemma \ref{l:realtree}.b.

\begin{lemma}\label{l:realtree} Let $M$ be a $(k-1)$-connected $2k$-manifold.

(a) If there is a $\Z_2$-embedding $h:K\to M$, then some modulo 2 intersection function for $K$ extends to a Gramian matrix of some $n$ homology classes in $H_k(M)$.
 
(b) If $M$ is orientable and there is a $\Z$-embedding $h:K\to M$, then some intersection function for $K$ extends to a Gramian matrix of some $n$ homology classes in $H_k(M;\Z)$.
\end{lemma}

\begin{addendum}[only used in \S\ref{ss:coro}]\label{a:disun}
(a/b) In Lemmas \ref{l:realtree}.a/b if $K=X\sqcup Y$ and both $X,Y$ are forestable / integrally forestable, 
then we may additionally assume that the block of the Gramian matrix corresponding to $X\times Y$ is zero.
\end{addendum}

We need the following essentially known definitions and lemmas.

{\it In the rest of this section $M$ is any connected $2k$-manifold, and $g:K\to M$ is any general position map.}

Take any pair of non-adjacent $k$-faces $\sigma,\tau$ of $K$.
By general position the intersection $g\sigma\cap g\tau$ consists of a finite number of points.
Assign to the pair $\{\sigma,\tau\}$ the residue
$$\nu(g)\{\sigma,\tau\}:=|g\sigma\cap g\tau|_2.$$
The obtained function $\nu(g):\t K/\Z_2\to\Z_2$ is called the (mod2) {\bf intersection cocycle} of $g$.

Let $\alpha$ be a $(k-1)$-face of $K$ disjoint from a $k$-face $\sigma$ of $K$.
An {\bf elementary coboundary} of the pair $(\alpha,\sigma)$ is the characteristic function of the subset $\delta_K(\alpha,\sigma)\subset \t K/\Z_2$
consisting of all pairs $\{\sigma,\tau\}$ with $\tau\supset\alpha$.
Functions $\nu,\nu':\t K/\Z_2\to\Z_2$ 
are called {\bf cohomologous} (mod2) if
$$
\nu-\nu'=\delta_K(\alpha_1,\sigma_1)+\ldots+\delta_K(\alpha_s,\sigma_s)
$$
for some $(k-1)$-faces $\alpha_1,\ldots,\alpha_s$ and $k$-faces $\sigma_1,\ldots,\sigma_s$ (not necessarily distinct).
Here $+$ is the componentwise addition. 

\begin{lemma}[see Remark \ref{r:hajolem}]\label{star} Let $\nu:\t K/\Z_2\to\Z_2$ be a function. 
There is a general position map $g':K\to M$ homotopic to $g$ and such that $\nu(g')=\nu$ if and only if $\nu$ is cohomologous to $\nu(g)$.
\end{lemma}


\begin{proof}[Proof of Lemma \ref{l:realtree}.a]
Take a $2k$-ball $B\subset\Int M-h(K)$.
By general position there is an embedding $K^{(k-1)}\to\partial B$.
Since $M$ is $(k-1)$-connected, $M-\Int B$ is also $(k-1)$-connected. 
Hence the restriction of $h$ to $K^{(k-1)}\to M-\Int B$ is null-homotopic. 
Then the restriction is homotopic to the composition $f_0$ of the embedding with the inclusion $\partial B\to M-\Int B$.
By the Borsuk Homotopy Extension Theorem this homotopy extends to a homotopy from $h$ to an extension $h':K\to M-\Int B$ of $f_0$.


Define a map $f:K\to B$ as $h'=f_0$ on $K^{(k-1)}$, and as the cone map over $h'|_{\partial\sigma}$ with a vertex in $\Int B$ on any $k$-face $\sigma$.
We take these vertices in general position.
Then $f$ is a general position map.
Let $y_\sigma:=[f\sigma\cup h'\sigma]\in H_k(M)$ for every $k$-face $\sigma$ of $K$.
Then these classes are as required because for any non-adjacent $k$-faces $\sigma,\tau$ we have
$$
y_\sigma\cap_My_\tau \stackrel{(3)}= |(f\sigma\cup h'\sigma)\cap(f\tau\cup h'\tau)|_2  \stackrel{(2)}=
|f\sigma\cap f\tau|_2 + |h'\sigma \cap h'\tau|_2 \stackrel{(1)}= |f'\sigma \cap f'\tau|_2. 
$$
Here

$\bullet$ equality (3) holds by general position and definitions of $y_\sigma$, $y_\tau$ and $\cap_M$; 

$\bullet$ equality (2) holds because $\sigma\cap\tau=\emptyset$ and $h'|_{K^{(k-1)}}$ is an embedding, so that
$f\sigma\cap h'\Int\tau\subset B\cap h'\Int\tau=\emptyset$, and analogously $f\tau\cap h'\Int\sigma=\emptyset$; 


$\bullet$ $f':K\to B$ is a general position map such that equality (1) holds for any non-adjacent $k$-faces $\sigma,\tau$ (i.e. such that $\nu(f')=\nu(f)+\nu(h')$);
such a map $f'$ exists by the case $M=B$ of Lemma \ref{star} since $\nu(h')$ is cohomologous to $\nu(h)=0$ by Lemma \ref{star}.
\end{proof}

\begin{lemma}[known]\label{l:hur} 
If $T$ is a $k$-complex such that $H_k(T;\Z)=0$ and $H_{k-1}(T;\Z)$ is free, then any map of $T$ to a $(k-1)$-connected space is homotopic to the map to a point.
\end{lemma}

\begin{proof}
If $k=1$, then $H_1(T;\Z)=0$ implies that $T$ is a forest, so the lemma is correct. 

Assume that $k\ge2$. 
Since $H_k(T;\Z)=0$ and $H_{k-1}(T;\Z)$ is free, by the Universal Coefficients Theorem  $H^k(T;G)=\Hom(H_k(T;\Z),G)\oplus\Ext(H_{k-1}(T;\Z),G)=0$ for any abelian group $G$.
Since $\dim T\le k\ge2$, homotopy classes of maps $T\to M$ to a $(k-1)$-connected space $M$ are in 1--1 correspondence with $H^k(T;\pi_k(M))=0$.
\end{proof}

\begin{proof}[Proof of Addendum \ref{a:disun}.a]
Take a $2k$-ball $B\subset\Int M-h(K)$.

Since $X$ is forestable, there is a maximal $k$-forests $T_X$ in $X$ such that $H_{k-1}(T_X;\Z)$ is free. 
Since $\dim T_X\le k$, the group $H_k(T_X;\Z)$ is free.
Since $H_k(T_X)=0$, by the Universal Coefficients Theorem it follows that $H_k(T_X;\Z)=0$.
Since $M$ is $(k-1)$-connected, $M-\Int B$ is also $(k-1)$-connected. 
Thus any map $T_X\to M-\Int B$ is null-homotopic by Lemma \ref{l:hur}.

Analogously one chooses $T_Y$ and proves that any map $T_Y\to M-\Int B$ is null-homotopic.

Then by the Borsuk Homotopy Extension Theorem $h$ is homotopic to a map $h':K\to M-\Int B$ such that $h'(K^{(k-1)})\subset\partial B$ and $h'(T_X\sqcup T_Y)$ is contained in a small neighborhood of $\partial B$.
By general position we may assume that $h'|_{K^{(k-1)}}$ is an embedding.

Repeat the second paragraph of the above proof of Lemma \ref{l:realtree}.a.
We obtain a map $f:X\sqcup Y\to B$ and classes $y_\sigma\in H_k(M)$.
For $Z=X,Y$ and any $k$-face $\sigma\subset Z$, set $\widehat\sigma :=0$ when $\sigma\subset T_Z$, and take any non-empty $k$-cycle modulo 2 \ $\widehat\sigma$ in $T_Z\cup\sigma$ when $\sigma\not\subset T_Z$.\arxiv{
\footnote{For $k=1$ the 1-cycle $\widehat\sigma$ is the union of $\sigma$ and the simple path in $T_Z$ joining the ends of the edge $\sigma$.
The $k$-cycle $\widehat\sigma$ is unique because $\sigma\subset\widehat\sigma$, hence for any other non-empty $k$-cycle $\widehat\sigma'$ in $T_Z\cup\sigma$ the $k$-cycle $\widehat\sigma+\widehat\sigma'$ is contained in $T_Z$, so is empty.}
}
Then
$$y_\sigma \overset{(1)}= [f\sigma\cup h'\sigma] \overset{(2)}=
[f\sigma\cup h'(\widehat\sigma+\sigma)]+[h'\widehat\sigma] \overset{(3)}=
0+h'_*\widehat\sigma = h_*\widehat\sigma\in H_k(M).$$
Here equality (1) is the definition, equality (2) holds because the $k$-cycle modulo 2 \ $f\sigma\cup h'(\widehat\sigma+\sigma)$ is contained in a neighborhood of $B$, and equality (3) holds because $h$ and $h'$ are homotopic.

So for any $k$-faces $\sigma\subset X$ and $\tau\subset Y$ we have
$$y_\sigma\cap_My_\tau = h_*\widehat\sigma\cap_M h_*\widehat\tau
= \sum_{\alpha\in \widehat\sigma, \beta\in \widehat\tau} |h\alpha\cap h\beta|_2 = 0,$$
where the last equality holds because $h$ is a $\Z_2$-embedding.
\end{proof}


Assume that $M$ is oriented.
Take any orientations on $k$-faces of $K$.
Assign to any ordered pair $(\sigma,\tau)$ of non-adjacent $k$-faces the integer
$$\nu_{\Z}(g)(\sigma,\tau):=g\sigma\cdot g\tau.$$
The obtained function $\nu_{\Z}(g):\t K\to\Z$ is called the (integer) {\bf intersection cocycle} of $g$.
This cocycle is {\bf $(-1)^k$-symmetric}, i.e. $\nu_{\Z}(g)(\sigma,\tau)=(-1)^k\nu_{\Z}(g)(\tau,\sigma)$.

Let $\alpha$ be an oriented $(k-1)$-face of $K$ disjoint from a $k$-face $\sigma$ of $K$.
An (integer) {\bf elementary coboundary} of the pair $(\alpha,\sigma)$ is the function $\delta_K(\alpha,\sigma):\t K\to\Z$ assigning
$$(-1)^k[\tau:\alpha]\text{ to }(\sigma,\tau),\quad [\tau:\alpha]\text{ to }(\tau,\sigma)\quad\text{and \quad 0 to any other pair},$$
where the \emph{incidence coefficient} $[\tau:\alpha]$ is defined e.g. in \cite[\S13.4]{FF89}, \cite[\S3]{HG}, \cite[\S10.6]{Sk20}. 
Functions $\nu,\nu':\t K\to\Z$ are called (integer) {\bf cohomologous} if
$$
\nu-\nu'=c_1\delta_K(\alpha_1,\sigma_1)+\ldots+c_s \delta_K(\alpha_s,\sigma_s)
$$
for some integers $c_1,\ldots,c_s\in\Z$, oriented $(k-1)$-faces $\alpha_1,\ldots,\alpha_s$, and $k$-faces $\sigma_1,\ldots,\sigma_s$ (not necessarily distinct).
Observe that change of the orientation of $\alpha$ forces change of the sign of $\delta_K(\alpha,\sigma)$.
Hence the cohomology equivalence relation does not depend on the orientations of $(k-1)$-faces.

\begin{lemma}\label{l:zstar} Let $\nu:\t K\to\Z$ be a $(-1)^k$-symmetric function, and $M$ an oriented $2k$-manifold.
There is a general position map $g':K\to M$ homotopic to $g$ and such that $\nu_{\Z}(g')=\nu$ if and only if $\nu$ is cohomologous to $\nu_{\Z}(g)$.
\end{lemma}

\begin{remark}\label{r:hajolem}
Lemmas \ref{star} and \ref{l:zstar} are essentially known.
They are due to van Kampen-Shapiro-Wu for $M=\R^{2k}$, see a proof in \cite[Lemma 1.5.8 and Proposition 1.5.9]{Sk18} (modulo 2 version for $k=1$) and in \cite[Lemma 3.5]{Sh57}, \cite{Wu58}, \cite[\S2]{FKT} (integer version for any $k$, cf. \cite[Remark 1.6.4]{Sk18}).
The proof for an arbitrary $M$ is analogous (e.g. the finger moves are done in a regular neighborhood of a path in $M$, which neighborhood is homeomorphic to the $2k$-ball).
The `only if' part of Lemma \ref{l:zstar} is \cite[Theorem 1]{Jo02}, \cite[Lemma 11]{PT19}, the `if' part of Lemma \ref{l:zstar} is \cite[Theorem 6]{Jo02} 
(in which the assumptions `$n\ge3$' and `$M$ is 1-connected' are  superfluous).
\end{remark}
 
\begin{proof}[Proof of Lemma \ref{l:realtree}.b and Addendum \ref{a:disun}.b]
The proofs are obtained from the above proofs of parts (a) by the following changes.
We take a $2k$-ball $B\subset\Int M-h(K)$ whose orientation is inherited from $M$.
Replace `$\Z_2$-embedding' by `$\Z$-embedding' and `forestable' by `integrally forestable'.   
Replace $f\sigma\cup h'\sigma$ by the integer $k$-cycle $f\sigma-h'\sigma$ that is the sum of $f\sigma$ whose orientation comes from the orientation of $\sigma$, and of $h'\sigma$ whose orientation comes from the opposite orientation of $\sigma$. 
Replace the displayed formula by 
$$
y_\sigma\cap_{M;\Z}y_\tau \stackrel{(3)}= (f\sigma-h'\sigma)\cdot(f\tau-h'\tau)  \stackrel{(2)}=
f\sigma\cdot f\tau + h'\sigma\cdot h'\tau \stackrel{(1)}= f'\sigma\cdot f'\tau. 
$$
Refer to Lemma \ref{l:zstar} instead of Lemma \ref{star}.

Besides, for Addendum \ref{a:disun}.b we define $\widehat\sigma$ differently (see below),
we have $H_k(T_X;\Z)=0$ by definition (without using $\dim T_X\le k$ and the Universal Coefficients Theorem),
and instead of $y_\sigma=h_*\widehat\sigma$ we only obtain $s_\sigma y_\sigma=h_*\widehat\sigma$ for some integer $s_\sigma\ne0$ (proved below).
Then for any $k$-faces $\sigma\subset X$ and $\tau\subset Y$ as in the above proof of Addendum \ref{a:disun}.a 
we have $s_\sigma s_\tau(y_\sigma\cap_{M;\Z}y_\tau)=0$, hence $y_\sigma\cap_{M;\Z}y_\tau=0$.

Take any maximal integer $k$-forest $T\subset K$.
For any $k$-face $\sigma\subset T$ set $\widehat\sigma=\widehat\sigma_{K,T}:=0$.
For any $k$-face $\sigma\subset K-T$ take any non-zero
integer $k$-cycle $\widehat\sigma=\widehat\sigma_{K,T}$ in $T\cup\sigma$.

If $\sigma\subset T$, then let $s_\sigma:=1$.  
Then $s_\sigma y_\sigma=0=h_*\widehat\sigma\in H_k(M;\Z)$. 

If $\sigma\subset K-T$, then let $s_\sigma$ be the coefficient of the oriented $\sigma$ in $-\widehat\sigma$.  
Then $s_\sigma\ne0$ and  
$$s_\sigma y_\sigma = [s_\sigma f\sigma-s_\sigma h'\sigma] =
[s_\sigma f\sigma-h'(\widehat\sigma+s_\sigma\sigma)]+[h'\widehat\sigma] = 0+h'_*\widehat\sigma = h_*\widehat\sigma\in H_k(M;\Z).\qedhere$$
\end{proof}


\begin{remark}\label{r:hiconn}
In the `only if' parts of Theorems \ref{t:matrix}, \ref{t:matrixz}, and in Lemma \ref{l:realtree}
the $(k-1)$-connectedness assumption can be replaced by either of the following assumptions 
(even for $k=1$, for which case the `can be replaced' is though trivial):

(i) every map $K^{(k-1)}\to M$ extendable to $K$ is null-homotopic;

(ii) $M$ is smooth $(k-1)$-parallelizable, i.e. on the union of $(k-1)$-faces of some triangulation of $M$ there are $2k$ tangent vector fields linearly independent at every point.

Condition (i) is sufficient because in the proof of Lemma \ref{l:realtree} instead of the $(k-1)$-connectedness it suffices to assume that $h|_{K^{(k-1)}}$ is null-homotopic.


Condition (ii) is sufficient by the following result (see Remark \ref{r:crit2}.b).

\emph{If $Q$ is a $k$-subcomplex of a smooth $(k-1)$-parallelizable $2k$-manifold $M$, then $Q$ is a subcomplex of some $(k-1)$-connected manifold $N$ such that $\cap_N$ is isomorphic to $\cap_M$.
The same holds with $\cap_N,\cap_M$ replaced by $\cap_{N;\Z},\cap_{M;\Z}$.}

This result is easily proved by surgery of $M$ below the middle dimension, not changing $\cap_M$ (or $\cap_{M;\Z}$), 
$(k-1)$-parallelizability \cite[Theorem 3]{Mi61}, and (by general position) performed outside $Q$.
\end{remark}

\section{Proofs of Proposition \ref{p:addi} and Corollaries \ref{c:rkcom}.bc}\label{ss:coro}

\begin{proof}[Proof of Proposition \ref{p:addi}.a]
The inequality
$P(X\sqcup Y) \le P(X)+P(Y)$ is clear.
We have
$$P(X\sqcup Y) \overset{(1)} = \rk M \overset{(2)}\ge \rk(X_+\cup Y_+) \overset{(3)}\ge \rk X_+ + \rk Y_+ \ge P(X)+P(Y), \quad\text{where}$$

$\bullet$ $M$ is a $2k$-manifold such that there is an embedding $h:X\sqcup Y\to M$ and equality (1) holds;

$\bullet$ $X_+$ and $Y_+$ are disjoint neighborhoods  of $h(X)$ and of $h(Y)$ in $M$;

$\bullet$ inequality (2) holds because $\rk$ is monotone.

$\bullet$ inequality (3) holds because $X_+$ and $Y_+$ are disjoint (in fact, (3) is an equality).
\end{proof}

\begin{proof}[Proof of Proposition \ref{p:addi}.b,c for $k\ge2$]
The following proof is analogous to (a).
Take a $2k$-manifold $M$ and a $\Z$- or $\Z_2$-embedding $h:X\sqcup Y\to M$ such that the equality (1) from the proof of (a) holds.
We may assume that $h$ is in general position, moreover, that away from the prescribed identifications along common faces, all self-intersections 
(including those of distinct adjacent faces) are isolated transverse double points. 
The required perturbation to general position preserves the signed or mod 2 intersection conditions for disjoint faces. 

Then there are regular neighborhoods $X_+$ and $Y_+$ of $h(X)$ and of $h(Y)$ in $M$ whose intersection is the union of some products $B^k\times B^k$, every product intersecting $h(X)$, $h(Y)$, $\partial X_+$ and $\partial Y_+$ by $B^k\times0$, $0\times B^k$, $B^k\times\partial B^k$ and $\partial B^k\times B^k$, respectively.

We give the remaining argument for (b); the argument for (c) is analogous.
Part (b) holds by the (in)equalities from the proof of (a) (although $X_+$ and $Y_+$ are no longer disjoint).
Take the homomorphisms $i_X,i_Y$ in $H_k$ induced by the inclusions $X_+,Y_+\to X_+\cup Y_+$.
The inequality (3) follows because for any $x,x'\in H_k(X_+)$ and $y,y'\in H_k(Y_+)$ we have
$$x \cap_{X_+} x' = i_Xx \cap i_Xx',\quad y \cap_{Y_+} y' = i_Yy \cap i_Yy'\quad\text{and}\quad i_Xx \cap i_Yy = 0.$$
Here $\cap:=\cap_{X_+\cup Y_+}$ and the first two equalities are obvious.
In the following paragraph we prove the latter equality.

The neighborhood $X_+$ is homotopy equivalent to $h(X)$.
The latter is obtained from $X$ by several identifications of pairs of points.
Then the map $h:X\to h(X)$ induces an isomorphism $H_k(X)\to H_k(h(X))$ because $k\ge2$.
(This is proved either by definition or by observing that $h(X)$ is homotopy equivalent to the union $U$ of $X$ and several arcs, and using the exact sequence of the pair $(U,X)$.)
Then $x$ is represented by $h$-image of a simplicial $k$-cycle $\overline x$ in $X$.
Analogous statement holds for $y$, $\overline y$, $Y$.
So
$$
i_Xx \cap i_Yy = \sum_{\sigma\in\overline x, \tau\in\overline y} |h\sigma\cap h\tau|_2 = 0,
$$
where the last equality holds because $h$ is a $\Z_2$-embedding.
\end{proof}

\arxiv{

{\bf Comment.}
(a) The following example shows that the above proof of (b,c) does not work for $k=1$.
Let  $M$  be the Klein bottle with a hole.

\emph{There is a $\Z_2$-embedding of the disjoint union  $K_3\sqcup K_3$  of two copies of  $K_3$ to $M$ with a certain number of holes such that the regular neighborhood of the image of each $K_3$ is homeomorphic to $M$.}

The surface $M$ is homeomorphic to the union of two (`upper' and `lower') Moebius bands which intersect by a segment in their boundaries.
Take two copies of $M$  intersecting by the union of two disks (the `plumbing' intersection of two `upper' Moebius bands, and the same for two `lower' Moebius bands).
Denote by  $M_+$  the union of the two copies.
Take a $\Z_2$-embedding of  $K_3$  on  $M$  having exactly one self-intersection point, which is self-intersection of an edge.
These two $\Z_2$-embeddings form together a $\Z_2$-embedding $h:K_3\sqcup K_3\to M_+$.

Propositions \ref{p:addi}.b,c for  $K_3\sqcup K_3$ are trivial.
However, the above proof of (b,c) does not work for $h$ (which is in general position).
Indeed, $M_+$ is homeomorphic to $M$ with a certain number of holes, so $\rk M_+ = 2 \ne 4 = 2\rk M$.

(b) \emph{For any $2k$-manifold $M$ there is a $k$-complex having no $\Z_2$-embedding to $M$
(and hence having neither embedding nor $\Z$-embedding to $M$).}

The case $k=1$ is known by \cite[Lemmas 6, 7]{SS13} or \cite[Theorem 1]{FK19}.
The general case holds by a $\Z_2$-version \cite[Theorem 1]{PT19} essentially proved in \cite{PT19}.

Alternatively, this fact follows from Proposition \ref{p:addi}.b and non-$\Z_2$-embeddability
of $\Delta_{2k+2}^k$ in any $2k$-manifold $M$ with trivial $\cap_M$.
Such a non-$\Z_2$-embeddability follows by Lemma \ref{l:realtree}.a together with Remark \ref{r:hiconn}.i, and the quantitative van Kampen-Flores Theorem, see survey \cite[\S4, ($VKF_d^+$) for $d=2k$]{Sk16}.
 
This proof shows that as an example one can take the disjoint union of $1+\rk M$ copies of $\Delta_{2k+2}^k$.
In other words, if the disjoint union of $s$ copies of $\Delta_{2k+2}^k$ (or $\Delta_{s(2k+3)-1}^k$) has a $\Z_2$-embedding to $M$, then $\rk M\ge s$.
For $\Delta_{s(2k+3)-1}^k$ this estimation is weaker than  \cite[Theorem 1(i)]{PT19} but stronger (asymptotically for $k\to\infty$) than \cite[Theorem 2]{GMP+}.
Analogously, proving \emph{even-rank additivity} recovers a weaker form of \cite[Theorem 1(ii)]{PT19}.

}

The following lemma is essentially known.
(In particular, the proof of (c) generalizes \cite[Proposition 4.5.11]{GS99}.)
We provide a proof for completeness.

\begin{lemma}\label{l:mrank1} Let $M$ be a $(k-1)$-connected $2k$-manifold.

(a) If $\cap_M$ is odd, then there is a $(k-1)$-connected $2k$-manifold $N_1$ such that $\rk\cap_{N_1}=1$.

(b) Let $A$ be the Gramian matrix of some homology classes in $H_k(M)$.
Then there is a $(k-1)$-connected $2k$-manifold $N$ such that $\cap_N$ has the same rank and type as $A$.

(c) For $M$ orientable let $A$ be the Gramian matrix of some homology classes in $H_k(M;\Z)$.
Then there is an orientable $(k-1)$-connected $2k$-manifold $N$ 
such that $\rk_\Z N=\rk A$ and $A$ is the Gramian matrix of some homology classes in $H_k(N;\Z)$. 
\end{lemma}
 
\begin{proof}
(a) \arxiv{\footnote{Part (a) follows from the PL analogue mentioned in Remark \ref{r:crit2}.c.
We present a simpler direct proof.}}\jonly{(See Remark \ref{r:crit2}.c and \cite[footnote 9]{Sk24}.)}
 If $k = 1, 2$, then $N_1 = \R P^2,\C P^2$ satisfy the requirements.

Assume that $k\geq3$.
Since $\cap_M$ is odd, there is $\alpha\in H_k(M)$ such that $\alpha\cap_M\alpha = 1$.
By Lemma \ref{l:realcl} $\alpha$ is represented by a map $f\colon S^k\to M$.
Since $k\ge3$, this map is homotopic to an embedding by the Penrose-Whitehead-Zeeman-Irwin Theorem \cite{Ir65}, cf. \cite[Theorem 2.9]{Sk06}.
So assume that $f$ is an embedding.
Let $N_1$ be the regular neighbourhood of $f(S^k)$.
Then $N_1$ is $(k-1)$-connected and $\rk\cap_{N_1}\leq\rk H_k(N_1) = \rk H_k(S^k) = 1$.
Since $\alpha\cap_M\alpha = 1$, it follows that $\rk\cap_{N_1} = 1$.

\smallskip
(b) If $A$ is even, then  $\rk A$ is even.
So let $N$ be the connected sum of $\rk A/2$ copies of $S^k\times S^k$.

If $A$ is odd, then $\cap_M$ is odd.
Hence by (a) there is a $(k-1)$-connected manifold $N_1$ such that $\rk\cap_{N_1} = 1$.
So let $N$ be the connected sum of $\rk A$ copies of $N_1$.

In both cases $\cap_N$ has the same rank and type as $A$.

\smallskip
(c, $k$ odd) Let $N$ be the connected sum of $s:=\rk A/2$ copies of $S^k\times S^k$.
Let $p_1,q_1,\ldots,p_s,q_s\in H_k(N;\Z)$ be the symplectic base.  
Take a unimodular matrix $U$ and normal form coefficients $a_1,\ldots,a_s$ given by Lemma \ref{l:in_sk_no_fo}. 
Then $U^TAU$ is the Gramian matix of homology classes $p_1,a_1q_1,\ldots,p_s,a_sq_s,0,\ldots0$, whose number is equal to the size of $A$. 
So $A$ is the Gramian matix of some homology classes in $H_k(N;\Z)$. 
  
\smallskip
(c, $k$ even) Take given $s$ homology classes $y_1,\ldots,y_s\in H_k(M;\Z)$.   
Take a link $\sqcup_i g_i:S^{k-1}\to S^{2k-1}$ with unknotted components $g_1,\ldots,g_s$, such that the linking number 
of $g_i$ and $g_j$ equals $A_{i,j}=y_i\cap_{M;\Z}y_j$ for every $1\le i<j\le s$.
A \emph{framing} of the link is an extension of the link to an embedding $\sqcup_i S^{k-1}\times D^k\to S^{2k-1}$.


If $k=2$, then take the framing $A_{i,i}=y_i\cap_{M;\Z}y_i$ of $g_i$ for every $i=1,\ldots,s$.

If $k\ge4$,  
then for every $i=1,\ldots,s$ define the framing of $g_i$ as follows. 
Represent the class $y_i\in H_k(M;\Z)$ by an embedding $y_i:S^k\to M$ (as in the proof of (a)).
Represent $S^k$ as the union of two $k$-disks $D^k_+$ and $D^k_-$ intersecting by $S^{k-1}$.
Represent a regular neighborhood of $y_i(S^k)$ as the union of two images of embeddings
$$e_{\pm,i}=e_\pm:D^k_\pm\times D^k\to M\quad\text{such that}\quad
e_\pm(S^{k-1}\times D^k) = e_+(D^k_+\times D^k)\cap e_-(D^k_-\times D^k).$$
Define the framing of $g_i$ to be the composition of 

$\bullet$ the autohomeomorphism $e_+^{-1}\circ e_-|_{S^{k-1}\times D^k}$ of $S^{k-1}\times D^k$ and

$\bullet$ a framing $S^{k-1}\times D^k\to S^{2k-1}$ of $g_i$ giving $S^k\times D^k$ as the result of the surgery of $B^{2k}$ along the framed knot.

Let $N$ be the result of the surgery of $B^{2k}$ along the obtained framed link.
For every $i=1,\ldots,s$ let $z_i\in H_k(N;\Z)$ be the homology class formed by 
the union of a disk $D^k_i\subset S^{2k-1}$ spanning $g_i(S^{k-1})$, and the core $(D^k\times0)_i$ of the $i$-th handle. 

These classes form a basis of $H_k(N;\Z)$ (this can be shown using MVS sequence). 
So $\rk_\Z N=\rk A$. 
Clearly, $z_i\cap_{N;\Z}z_j = y_i\cap_{M;\Z}y_j=A_{i,j}$ when $i\ne j$. 
The capped core sphere representing $z_i$ has the same tubular neighborhood as the sphere representing $y_i$. 
Consequently its normal Euler number is $z_i\cap_{N;\Z}z_i = y_i\cap_{M;\Z}y_i=A_{i,i}$.
Thus the Gramian matrix of $z_1,\ldots,z_s\in H_k(M;\Z)$ is $A$. 
Thus $\cap_{N,\Z}$ has the matrix $A$ in the base $z_1,\ldots,z_s$.  
\end{proof}

 
\begin{proof}[Proof of Corollary \ref{c:rkcom}.b]
The inequality $r_{\Z_2}(X\sqcup Y) \le r_{\Z_2}(X)+r_{\Z_2}(Y)$ is clear (because the connected sum of $(k-1)$-connected $2k$-manifolds is $(k-1)$-connected by Hurewicz theorem and the analogous assertions for simple connectedness and for homology connectedness).
The opposite inequality follows by
$$r_{\Z_2}(X\sqcup Y) \stackrel{(1)}= \rk M \stackrel{(2)}\ge \rk A \stackrel{(3)}= \rk A_X+\rk A_Y \stackrel{(4)} \ge r_{\Z_2}(X)+r_{\Z_2}(Y).$$
Here

$\bullet$ $M$ is a $(k-1)$-connected $2k$-manifold such that there is a $\Z_2$-embedding $h:X\sqcup Y\to M$ and equality (1) holds;

$\bullet$ $A$ is the Gramian matrix given by Lemma \ref{l:realtree}.a;
 
$\bullet$ inequality (2) holds by Lemma \ref{l:gram};

$\bullet$ $A_X$ and $A_Y$ are the restrictions of $A$ to $k$-faces of $X$ and of $Y$, respectively;

$\bullet$ equality (3) holds by Addendum \ref{a:disun}.a.
  
$\bullet$ inequality (4) follows from the inequalities $\rk A_X \ge r_{\Z_2}(X)$ (proved in the following paragraph) and $\rk A_Y \ge r_{\Z_2}(Y)$ (proved analogously).

By Lemma \ref{l:realtree}.a $A_X$ extends some modulo 2 intersection function of $X$.
By Lemma \ref{l:mrank1}.b there is a $(k-1)$-connected $2k$-manifold $N$ such that $\cap_N$ has the same rank and type as $A_X$.
So by Theorem \ref{t:matrix} there is a $\Z_2$-embedding $X\to N$.
Then $\rk A_X = \rk N \ge r_{\Z_2}(X)$.
\end{proof}

\begin{proof}[Proof of Corollary \ref{c:rkcom}.c]
The proof is obtained from the proof of Corollary \ref{c:rkcom}.b by the following changes.
Replace

$\bullet$ $r_{\Z_2}$ by $r_{\Z}$, \ $\rk M$ by $\rk_\Z M$, \ $\cap_M$ by $\cap_{M;\Z}$;

$\bullet$ `$\Z_2$-embedding' by `$\Z$-embedding';
 
$\bullet$ references to Lemma \ref{l:realtree}.a and Addendum \ref{a:disun}.a by references to Lemma \ref{l:realtree}.b and Addendum \ref{a:disun}.b;

$\bullet$ `$(k-1)$-connected' by `orientable $(k-1)$-connected';

$\bullet$ the last paragraph by the following paragraph.

By Lemma \ref{l:mrank1}.c there is an orientable $(k-1)$-connected $2k$-manifold $N$ such that $\rk_\Z N=\rk A_X$, and $A_X$ is the Gramian matrix of some homology classes in $H_k(N;\Z)$.
So by Theorem \ref{t:matrixz} there is a $\Z$-embedding $X\to N$.
Then $\rk A_X = \rk_\Z N \ge r_\Z(X)$.
\end{proof}

\arxiv{

\section{Appendix: unreliability of some papers on embeddings}\label{s:unre}


The following Remark \ref{r:crit}.ab justifies that the proofs of \cite{SS13} are not \emph{reliable} (as explained in \cite[\S1]{Sk21d}).
This remark is an objective 
result of the scientific discussion, not merely the author's opinion. 
Indeed, I invited (see Remark \ref{r:let}) the authors of the criticized papers to present their public reply to the criticism,  
and I promised to present (a reference to) the reply even if I would not agree with it. 
No public reply was received. 
 
\smallskip
{\bf Comment.}
The purpose of Remark \ref{r:crit}.ab is (not fight for priority but) suggesting to provide reliable proofs (preferably by the authors) to enable people to use the results and methods of \cite{SS13}.
I suggested to the authors to make improvements upon these critical remarks, so that I could refer to an arXiv update of their paper, and omit critical remarks on the earlier version, see \cite[Remark 2.3]{Sk21d}.
From our discussions with M. Schaefer I conclude that he does not plan to produce a rigorous proof that lives up to the reliability standards pursued in Remark \ref{r:crit}.ab.
This invites other authors to do that.  
Still, I would call the block additivity stated in \cite{SS13} `the result of Schaefer-Stefankovi\v c, with some details fixed in such and such paper', cf.  Remark \ref{r:mot}.a.
I am not stating that it is impossible or hard to make corrections corresponding to Remark \ref{r:crit}.ab, cf.  \cite[Remark 2.3.b]{Sk21d} and Remark \ref{r:crit}.d.



\begin{remark}\label{r:crit}
\textbf{(a)} In \cite{SS13} p. 2, the definition of $y_e$ is not mathematically rigorous (and is unclear).
The expression \emph{`$e$ is pulled through the $i$-th crosscap an odd number of times'} has no rigorous meaning.
It is not even written (and not clear) what is defined:

\quad (i) a vector $y_e$ for \emph{any} map of the graph to the sphere with $s$ Moebius films (`surface with $s$ crosscaps' in the terminology of \cite{SS13});

\quad (ii) a construction of \emph{some} map of the graph to the sphere with $s$ Moebius films, starting from vectors $y_e$ and given map (=drawing) $D$ from the graph to the plane (and from vectors $x_e$, but even the case when  all $x_e$ are zeros or not mentioned has the described problem);

\quad (iii) something else, see \cite[Remark 2.1.c]{Sk21d}.

In case (i) the vector $y_e$ cannot be defined without first

\quad\quad (N) taking specific representation the sphere with $s$ Moebius films as the union of the sphere with $s$ holes, and making the embedding `nice' w.r.t. this specific representation.

Indeed,

$\bullet$ any non-self-intersecting curve on a surface has a neighborhood homeomorphic to the disk, so
no embedded edge `is pulled through any crosscap' unless we have the above specific representation;

$\bullet$ an edge is not a closed curve, so the number of times `the edge is pulled through a crosscap' is not defined (e.g. it is not written, and it is not clear, how to define this number when one end of the edge is on the crosscap, and the other end is outside the crosscap).

But (N) is not done in \cite[p. 2]{SS13}.

In case (ii) proofs of \cite[Lemmas 3 and 4]{SS13} do not work because the input there is \emph{any} map of the graph to the sphere with $s$ Moebius films.

\emph{The vector $y_e$ is the main object of \cite{SS13}, so the above critical remark affects the whole paper.}
E.g. the statements of Lemmas 3 and 4 have no rigorous mathematical meaning (and are unclear).

\textbf{(b)} In \cite[p. 2]{SS13} the following result is not rigorously stated (and is unclear):
\emph{`This definition is equivalent to the more intuitive definition given in the introduction (see, for example, Levow [5, Theorem 3]).}'
Leaving aside that `this definition' is not mathematically rigorous (and is unclear) as explained in (a), the
`equivalent' is not defined (and is unclear).
The cited result [5, Theorem 3] does not have any equivalence in its statement.


After (or before) the above-quoted unclear sentence of \cite[p. 2]{SS13} it is not written that
\cite[Remark 1]{SS13} contains an attempt for rigorous formulation corresponding to that sentence:

(*) `\emph{If $D$ is a drawing of a graph $G$ in some surface $S$, then there is a $\Z_2$-drawing $(D',x,y)$ of $G$ in $S$ so that $i_D(e,f)=i_{D',x,y}(e,f)$ for every pair $(e,f)$ of independent edges.}'.

This is not mathematically rigorous (and is unclear) because `a $\Z_2$-drawing of $G$ in $S$' is not defined.
Instead of reading a proof or a reference to a proof of (*), one reads in \cite[Remark 1]{SS13}:
`\emph{As mentioned earlier, a result like this (with a slightly different model) was stated by Levow [5].}'

Presumably `of `$G$ in $S$' should be replaced by `of $G$ in the plane'.
Unravelling the definitions and getting rid of $x$ (because $x$ can be realized by change of $D'$), we obtain the following corrected rigorous version of (*).


{\bf Algebraization  Lemma.} \emph{If $D$ is a drawing of a graph $G$ in some surface $S$, then there are a drawing $D'$ of $G$ in the plane and a vector $y\in\Z_2^E$ such that $i_D(e,f)=i_{D'}(e,f)+y_e^Ty_f$ for every pair $(e,f)$ of independent edges.}


This is the main basic result used in \cite{SS13}, so the above critical remark affects the whole paper.

\textbf{(c)} (consequences for \cite{FK19})
The paper \cite{FK19} uses \cite[Lemmas 3 and 4]{SS13}.
Thus the proofs of \cite{FK19} are unreliable.

Moreover, the Algebraization Lemma is stronger than (the estimation `$\ge$' of)
\cite[Corollary 10 (the first sentence)]{FK19}.\footnote{Indeed, the Algebraization Lemma for a $\Z_2$-embedding $D$, and known \cite[Lemmas 6 and 7]{FK19} trivially imply that result of \cite{FK19}.
Modulo those lemmas and a choice of basis in the homology group, that result of \cite{FK19} is Lemma \ref{l:realtree}.a for $k=1$ (due to Bikeev-Fulek-Kyn{\v{c}}l).}
The result \cite[Corollary 10 (the first sentence)]{FK19} uses \cite[Proposition 9]{FK19} which uses \cite[Lemma 5]{FK19} which is the same as \cite[Lemma 4]{SS13} which uses the Algebraization Lemma.
This looks like a vicious circle.
This is yet another motivation for appearance of reliable proofs (although I know how to rewrite the argument to avoid vicious circle).

\textbf{(d)} (recovery of results)
The results of \cite{FK19} are recovered by \cite[Theorem 1.1]{Bi21} (partly attributed to Fulek-Kyn{\v{c}}l).

The Algebraization Lemma is recovered by the implication ($\Rightarrow$) of \cite[Theorem 1.1.b]{Bi21} (modulo the known algebraic lemmas, and for $Y=(y_e)$), or by the proof of this implication in \cite[\S2]{Bi21}, or alternatively by Lemma \ref{l:realtree}.a for $k=1$ (whose proof is slightly different from \cite[\S2]{Bi21}).
Hopefully the block additivity of \cite{SS13} can also be recovered using
\cite[Theorem 1.1]{Bi21} (obtained using ideas of \cite{FK19}), cf. Remark \ref{r:mot}.c.
Also, I have an idea of how to rewrite papers \cite{SS13, FK19} to make the proofs reliable and closer to the initial ideas of \cite{SS13, FK19} than to \cite{Bi21}.
Nothing of these makes the proofs of \cite{SS13, FK19} reliable.
\end{remark}

\small

\begin{remark}\label{r:let}
Here I present my letters to D. \v Stefankovi\v c (analogous letter was earlier sent to M. Schaefer; Cc M. Schaefer), to R. Fulek and J. Kyn\v cl.

\textbf{(a)} (Sep 2, 2024) Dear Daniel,

Hope you are fine and healthy.

The following message might seem to you way too direct and formal.
If so, could you shortly contact Marcus on our latest discussion with him, of which this letter is a user-oriented result
(and on our discussions with him throughout the years).
Hopefully we could treat these questions in a professional way.

My purpose is to give a proper introduction to  arXiv:2112.06636.
In particular,

* to give proper credit to papers on $\Z_2$-embeddings of graphs (see Remarks 1.1.6.abcd).

* to inform a reader if some published proofs are not reliable (see arXiv:2101.03745).

Attached please find the planned update of arXiv:2112.06636.
See Remark 1.1.8a and \S2.5.
I would be grateful if you could show that some specific sentences of Remark 2.5.1.ab are incorrect (or unclear).
I am willing to make corrections corresponding to your remarks.

Could you let me know if you plan (in the forthcoming future) to make improvements upon Remark 2.5.1?

If yes, then I would be glad to refer to arXiv update of your paper, and omit critical remarks on the earlier version (see arXiv:2101.03745, Remark 2.3).

If no, then (possibly after working on your remarks and sending you the revision), I would be glad to publish your and Marcus' reply to (possibly modified) Remark 2.5.1, or a reference to such a reply.

I would be glad to do that even if I disagree with your reply.
Please send me your text for publication, stating that you perpetually release copyright for this text
(or a reference to a text involving `perpetual release of copyright' statement).

Best Regards, Arkadiy.

\textbf{(b)} (Sep 8, 2024) Dear Rado, Dear Jan (Honza)

Hope you are fine and healthy.

My purpose is to give a proper introduction to  arXiv:2112.06636.
In particular,

* to give proper credit to papers on $\Z_2$-embeddings of graphs (see Remarks 1.1.8.abcd).

* to inform a reader if some published proofs are not reliable (see arXiv:2101.03745).

Attached please find the planned update of arXiv:2112.06636.
See Remark 1.1.8a and \S2.5.
Remark 2.5.1.c concerns your paper [FK19].

Could you let me know if you plan (in the forthcoming future)
to update arXiv version of [FK19] to make it independent of [SS13]?
Your paper [FK19] is written so nicely that it would be easy for you to do that.
You implicitly prove the Algebraization Lemma.
It would be easy for you to explicitly state and prove its improvement required for lemmas from  [SS13] used in your paper, and to reprove the lemmas.
If you are in a hurry, then you can just refer to arXiv:2012.12070v2 for a rigorous proof of the required $\Z_2$-embeddability criterion (which is partly attributed to you there).
The main result of [FK19] is anyway not the $\Z_2$-embeddability criterion, but its application to the quadratic estimation for $K_{n,n}$.

If yes, then I would be glad to refer to arXiv update of your paper, and omit critical remarks on the earlier version (see arXiv:2101.03745, Remark 2.3).

If no, then I would be glad to publish your reply to Remark 2.5.1.c, or a reference to such a reply.

I would be glad to do that even if I disagree with your reply.
Please send me your text for publication, stating that you perpetually release copyright for this text
(or a reference to a text involving `perpetual release of copyright' statement).

Best Regards, Arkadiy.

\textbf{(c)} (Sep 15, 2024) Dear Daniel, Dear Marcus, Dear Rado, Dear Jan (Honza),

Attached please find the planned update of arXiv:2112.06636 [added later: this is v4].
See Remark 1.1.8a and \S2.5.

On one hand, the readers would be grateful to see your reply to Remark 2.5.1.
In order to publish your reply on arXiv I need your text for publication, and your statement that you perpetually release copyright for this text (or a reference to a text involving `perpetual release of copyright' statement).

On the other hand, I will not inform you on further development unless you express your interest.

Best Regards, Arkadiy.
\end{remark}

}

{\it Books, surveys and expository papers in this list are marked by the stars.}

\smallskip
Moscow Institute of Physics and Technology,
Independent University of Moscow.
Email: \texttt{skopenko@mccme.ru}.
\linebreak
\texttt{https://users.mccme.ru/skopenko/}.

\end{document}